\documentclass[reqno]{amsart}
\usepackage[T1]{fontenc}
\usepackage[latin1]{inputenc}
\usepackage{amssymb}
\usepackage{amsfonts}
\usepackage{amsmath,mathtools}
\usepackage{mathrsfs}
\usepackage{graphicx}
\usepackage{color}
\usepackage{hyperref}
\usepackage{esint}
\usepackage{verbatim}
\usepackage{bbm}
\usepackage{tikz}
\usepackage{setspace}
\usepackage{stackrel}
\usepackage{soul}

\newcommand{\N}[1]{\mathbb{N}^{#1}}
\newcommand{\R}[1]{\mathbb{R}^{#1}}
\renewcommand{\S}[1]{\mathbb{S}^{#1}}

\newcommand{\cC}{\mathcal C}
\newcommand{\cD}{\mathcal D}
\newcommand{\cE}{\mathcal E}
\newcommand{\cF}{\mathcal F}

\newcommand{\cH}{\mathcal H}

\newcommand{\cL}{\mathcal L}
\newcommand{\cM}{\mathcal M}
\newcommand{\cO}{\mathcal O}

\newcommand{\cS}{\mathcal S}

\newcommand{\ED}{\mathcal{ED}}

\newcommand{\bulk}{\mathrm{bulk}}

\newcommand{\de}{\mathrm d}

\newcommand{\rec}{\mathrm{rec}}

\newcommand{\surface}{\mathrm{surf}}

\newcommand{\eps}{\varepsilon}

\newcommand{\tr}{\operatorname{tr}}

\renewcommand{\geq}{\geqslant}
\renewcommand{\leq}{\leqslant}

\newcommand{\longrightharpoonup}{\relbar\joinrel\rightharpoonup}
\newcommand{\wto}{\rightharpoonup}

\newcommand{\wsto}{\stackrel{*}{\rightharpoonup}}

\newcommand{\weakst}{\stackrel{*}{\longrightharpoonup}}
\newcommand{\weakstH}{\stackrel[H]{*}{\longrightharpoonup}}
\newcommand{\pweak}{\stackrel{p}{\longrightharpoonup}}

\newcommand{\average}{{\mathchoice {\kern1ex\vcenter{\hrule
height.4pt width 8pt depth0pt}
\kern-11pt} {\kern1ex\vcenter{\hrule height.4pt width 4.3pt
depth0pt} \kern-7pt} {} {} }}

\newcommand{\res}{\mathop{\hbox{\vrule height 7pt width .5pt depth
0pt\vrule height .5pt width 6pt depth0pt}}\nolimits}

\mathchardef\emptyset="001F

\providecommand{\U}[1]{\protect\rule{.1in}{.1in}}
\numberwithin{equation}{section}
\newtheorem{definition}{Definition}[section]
\newtheorem{theorem}[definition]{Theorem}
\newtheorem{lemma}[definition]{Lemma}
\newtheorem{proposition}[definition]{Proposition}
\newtheorem{assumptions}[definition]{Assumptions}

\newtheorem{corollary}[definition]{Corollary}

\newtheorem{example}[definition]{Example}

\theoremstyle{definition} {\newtheorem{remark}[definition]{Remark}}

\makeindex

\title[Hierarchical Structured Deformations via the Global Method for Relaxation]{A Comprehensive Approach via Global Relaxation to the Variational Modelling of Hierarchical Structured Deformations}%
\author[A.~C.~Barroso]{Ana Cristina Barroso}
\address[A.~C.~Barroso]{Departamento de Matem\'atica and CEMS.UL, 
Faculdade de Ci\^encias da Universidade de Lisboa,
Campo Grande, Edif\' \i cio C6, Piso 1,
1749-016 Lisboa, Portugal}
\email{acbarroso@ciencias.ulisboa.pt}
\author[J.~Matias]{Jos\'{e} Matias}
\address[J.~Matias]{ Centro de An\'{a}lise Matem\'{a}tica Geometria e Sistemas Din\^{a}micos, Departamento de Matem\'atica, Instituto Superior T\'ecnico, Av.~Rovisco Pais, 1, 1049-001 Lisboa, Portugal}
\email{jose.c.matias@tecnico.ulisboa.pt}
\author[M.~Morandotti]{Marco Morandotti}
\address[M.~Morandotti]{Dipartimento di Scienze Matematiche ``G.~L.~Lagrange'', Politecnico di Torino, Corso Duca degli Abruzzi, 24, 10129 Torino, Italy}
\email{marco.morandotti@polito.it}
\author[D.~R.~Owen]{David R.~Owen}
\address[D.~R.~Owen]{Department of Mathematical Sciences, Carnegie Mellon University, 5000 Forbes Ave., Pittsburgh, PA 15213 USA}
\email{do04@andrew.cmu.edu}
\author[E.~Zappale]{Elvira Zappale}
\address[E.~Zappale]{Dipartimento di Scienze di Base ed Applicate per l'Ingegneria, Sapienza - Universit\`{a} di Roma, Via Antonio Scarpa, 16, 00161 Roma, Italy and
CIMA, Universidade de \'Evora, Portugal}
\email{elvira.zappale@uniroma1.it}

\date{\today}


\makeindex

\begin{document}

\begin{abstract}
The response of many materials to applied forces and boundary constraints depends upon internal geometric changes at multiple submacroscopic levels. Hierarchical structured deformations provide a mathematical setting for the description of such changes and for the variational determination of the corresponding energetic response. The research in this article provides substantial refinements and broadenings of the mathematical setting both for the underlying geometrical structure and for the variational analysis of energetic response. The mathematical tools employed in this research include the global method for relaxation and establish the equivalence of a relaxed energy obtained via relaxation under simultaneous geometrical changes at all levels and a relaxed energy obtained via iterated relaxations proceeding from the deepest submacroscopic level successively to the macroscopic level.
\end{abstract}

\maketitle

\allowdisplaybreaks

\tableofcontents


\section{Introduction}
The theory of (first-order) structured deformations, introduced by Del Piero and Owen in \cite{DPO1993}, proposes a general framework in the context of continuum mechanics to study deforming bodies, without committing at the onset to a specific mechanical theory, such as elasticity, plasticity, or fracture. 
It does so by considering both macroscopic and sub-macroscopic phenomena in the description of a deformation.
These sub-macroscopic phenomena include, for instance, slips and separations within the lattice of crystalline materials, and are referred to as \emph{disarrangements} in \cite{DPO1993}. 
In traditional macroscopic descriptions, a single field $g$ (and its gradient $\nabla g$) suffice to characterize the deformation of a continuous body. 
The theory of structured deformations introduces an additional geometrical field $G$, of the same tensorial character as $\nabla g$, to account for smooth sub-macroscopic changes,
while the difference $\nabla g - G$ captures the effects of sub-macroscopic disarrangements. 
Hence, two distinct objects are needed to describe the deformation of a continuous body, and a \emph{structured deformation} can be identified with a pair $(g, G)\in SD(\Omega)$, where
\begin{equation}\label{eq_SD}
SD(\Omega) \coloneqq SBV(\Omega;\R{d}) \times L^1(\Omega;\R{d\times N}),
\end{equation}
$SBV(\Omega ;\R{d})$ is the space of special functions of bounded variation on $\Omega $ with values in $\R{d}$, $L^{1}(\Omega ;\R{d\times N})$ is the space of integrable functions on $\Omega $ with values that are $d\times N$ matrices, and $\Omega\subset\R{N}$ is a bounded, connected, open set describing the reference configuration of the body. 
The variational formulation proposed by Choksi and Fonseca \cite{CF1997} relies on an energetic approach to solve the issue of assigning an energy to a structured deformation, with a view toward the application of the direct method of the calculus of variations to find equilibrium configurations of deforming bodies. 
Given an \emph{initial energy} 
\begin{equation}\label{CF} 
E(u) \coloneqq \int_\Omega W(\nabla u(x))\, \de x + \int_{\Omega \cap S(u)} \psi([u](x), \nu_u(x))\, \de\mathcal H^{N-1}(x)
\end{equation}
defined on functions $u\in SBV(\Omega;\R{d})$, the energy $I(g,G)$ assigned to a structured deformation is the one obtained in the energetically most economical way to approximate the pair $(g,G)$ by deformations $u_n\in SBV(\Omega;\R{d})$ in the following sense
\begin{equation}\label{eq_convergence} 
u_n \to g \; \text{in}\; L^1 (\Omega; \mathbb R^d), \qquad \nabla u_n \wsto G\; \text{in}\; \mathcal M( \Omega; \mathbb R^{d\times N}).
\end{equation}
Mathematically speaking, the energy $I$ is defined via relaxation, according to
\begin{equation}\label{eq_relaxation_CF}
I(g,G) \coloneqq \inf \Big\{ \liminf_{n\to +\infty} E(u_n) : \text{$\{u_n\}\subset SBV(\Omega;\R{d})$ converging to $(g,G)$  in the sense of \eqref{eq_convergence}} \Big\}
\end{equation}

Since their introduction in \cite{DPO1993} and their study in the context of the variational theory established in \cite{CF1997}, there has been an extensive body of work concerning both mechanical applications of structured deformations \cite{CDPFO1999,DP2001,DP2018,DPO1995,DO2000,DO2002,DO2002a,DO2003,F2007,FRC2005,FRC2010,FRC2016,GMK2024,O2017,OP2000,OP2015,RCL2023}  and mathematical approaches investigating the energy $I$ and its properties in a variety of contexts, including \cite{AMMZ,BMS2012,nogap,CMMO,FHP,KKMZ,L2000,MMZ,S2015,S17}.
In particular, the results of \cite{CF1997} were generalized to account for continuous spatial dependence of the energy densities~$W$ and~$\psi$ in \eqref{CF}, see, e.g., \cite{BMMO2017,book,MMOZ} and \cite{BMZ2024} for  discontinuous~$W$.

Motivated by the fact that different natural \cite{BBRC2022,Bertoldi_etal,ChenPugno,Gibson,HabibiLu,Lakes,Launey_etal,Ozcoban_etal,WangGupta} and engineered materials \cite{ChenPugno,Gu_etal,Wegst_etal} exhibit multiple levels of sub-macroscopic behavior, Deseri and Owen \cite{DO2019} proposed a generalization of  the theory to account for the effects of such different sub-macroscopic scales, introducing the concept of a \emph{hierarchical structured deformation}.
In this context, given $L \in \mathbb N\setminus\{0\}$, an $(L+1)$-hierarchical structured deformation has $L$ different levels of sub-macroscopic behavior, accounted for by different fields $G_i$\,, ($i = 1, \ldots, L$). A structured deformation ($L=1$) in the original sense of \cite{DPO1993}, corresponds to a \emph{two-level hierarchical structured deformation} in the sense of \cite{DO2019}.

This work departs from two previous articles which propose a variational approach to hierarchical structured deformations \cite{BMMOZ,BMZ2024} and has a two-fold objective.
On the one hand, it extends the known results for two-level structured deformations to the case under consideration: we obtain here the Approximation Theorem (see Theorem~\ref{appTHMh}) for hierarchical structured deformations $(g,G_1,\ldots,G_L)$ (see Definition~\ref{Def2.1}), and we provide asymptotic cell formulae for the relaxed bulk and surface energy densities (see \eqref{f2p} and \eqref{Phi2p}).
The novelty in this paper rests on the fact that we achieve these results under fairly weak regularity assumptions on the initial bulk energy density; in particular, we consider a Carath\'{e}odory function $(x,A)\mapsto W(x,A)$, thus allowing us to account for non-homogeneous behavior of the body.
Under this weaker assumption, we stress that Theorem \ref{thm_approx_CF} generalizes all previous results pertaining to the variational formulation of two-level structured deformations in the sense of \cite{CF1997}.

On the other hand, in proving the integral representation result in the specific case of three-level hierarchical structured deformations $(g,G_1,G_2)$ (see Theorem~\ref{mainL}), we set the basis for the extension to the general case $L>2$ (the case $L=3$ is addressed in Appendix~\ref{app_L=3}).
We adopt here an approach based on the global method for relaxation \cite{BFM1998}, that follows and adapts the results in~\cite{BMZ2024,BMZGaeta}, as opposed to~\cite{BMMOZ}, where the assignment of an energy to a hierarchical structured deformation is established by means of a recursive relaxation scheme. 
In that case, departing from the initial energy \eqref{CF}, the full relaxation is obtained by performing $L$ partial relaxations, by upscaling, step by step, from the most sub-macroscopic level to the macroscopic one.
Each of these partial relaxations is performed \emph{\`{a} la} Choksi-Fonseca, using the machinery of \cite[Theorems 2.16 and 2.17]{CF1997}, upon verifying that the partially relaxed energy densities that are obtained at each step still verify the hypotheses that allow for the next step to be carried out successfully.

As already pointed out in \cite{BMMOZ}, it is key both for its mathematical interest and its mechanical one to perform an all-at-once relaxation, which is suggested by our Approximation Theorem~\ref{appTHMh}. 
Indeed, the relaxation
\begin{equation*}
\begin{split}
I(g,G_1,\ldots,G_L)\coloneqq \inf\Big\{\liminf_{n_1}\cdots\liminf_{n_L} E\big(u_{n_1,\ldots,n_L}\big):&\, (u_{n_1,\ldots,n_L})\subset SBV(\Omega;\R{d}),\\
&\,\, u_{n_1,\ldots,n_L} \weakstH (g,G_1,\ldots,G_L)\Big\}
\end{split}
\end{equation*}
is well defined, as the set of approximating sequences is not empty. 
This procedure, carried out in this work under the rather general set of Assumptions~\ref{gen_ass}, leads, in principle, to a relaxed energy  which is, in general not greater than that obtained in \cite{BMMOZ} via the recursive relaxation. 
However, comparing the explicit example \cite[Section~3.3]{BMMOZ} with our result, we conclude that there is at least one case where the relaxed energies coincide (see Example~\ref{example}).

We point out that the approaches in \cite{BMMOZ} and herein, both can be deduced as iterated schemes inspired by \cite{CF1997}; the main difference is that the present method does not freeze any field appearing in the previous steps, since our iterated scheme is fully equivalent to a relaxation where all the levels of deformations iteratively converge.

Departing from the fact that the Approximation Theorem can indeed be viewed as an iterated scheme (see Corollary~\ref{S004} and Proposition~\ref {T003}), but without fixed targets at each step as in \cite{BMMOZ}. 
We note that some of the properties of the energy densities are preserved, while others change after the first iteration step, as illustrated in Theorem~\ref{thm_propdens}. 
We close this introduction by mentioning that Theorem~\ref{globalmethod-new}, which is established for a general $L\in\mathbb{N}\setminus\{0\}$, is the key ingredient for the proof of the main result.

\smallskip

The plan of the paper is the following: in Section~\ref{prelim}, we establish the notation, define the space of $(L+1)$-level hierarchical structured deformations, prove the Approximation Theorem, and include some auxiliary results; in Section~\ref{sec_gen_CF}, we list the standing assumptions on the initial energy densities $W$ and $\psi$ and we generalize the relaxation result from \cite{CF1997} to our setting; in Section~\ref{multi}, we state and prove the relaxation result for the case $L=2$, whereas we postpone to Appendix~\ref{app_L=3} an illustration of the result for $L=3$.

\color{black}
\section{Preliminaries}\label{prelim}

\subsection{Notation and main definitions} 
Throughout the manuscript, we will use the following notations. 
\begin{itemize}
	\item $\Omega \subset \mathbb R^{N}$ is a bounded, connected open set with Lipschitz boundary. 
	$\mathbb S^{N-1}$ is the unit sphere in $\mathbb R^N$; 
	\item (cubes) $Q\coloneqq (-\tfrac12,\tfrac12)^N$ denotes the open unit cube of $\mathbb R^{N}$ centered at the origin; for any $\nu\in\mathbb S^{N-1}$, $Q_\nu$ denotes any open unit cube in $\mathbb R^{N}$ with two faces orthogonal to $\nu$;  for any $x\in\mathbb R^{N}$ and $\delta>0$, $Q(x,\delta)\coloneqq x+\delta Q$ denotes the open cube in $\mathbb R^{N}$ centred at $x$ with side length $\delta$ and $Q_\nu(x,\delta)\coloneqq x+\delta Q_\nu$;
	\item ${\mathcal O}(\Omega)$ is the family of all open subsets of $\Omega $;
	\item $SD(\Omega)$, $SD_{L,p}(\Omega)$, and $SD_L(\Omega)$ (see \eqref{eq_SD} and Definition~\ref{Def2.1}), with $L\in\mathbb{N}\setminus\{0\}$ and $p\geq1$, are spaces of (hierarchical) structured deformations;
	\item unless otherwise specified, $C>0$ represents a generic constant that may change from line to line.
\end{itemize}

\medskip 

The variational theory of structured deformations finds its natural setting in the space $SBV(\Omega;\R{d})$; we refer the reader to \cite{AFP} for a general exposition on $(S)BV$ functions.

A fundamental result in the theory of structured deformations is the Approximation Theorem \cite[Theorem~5.8]{DPO1993}, a counterpart of which was recovered in \cite[Theorem~2.12]{CF1997} in the $SBV$ framework and in \cite{KKMZ,S2015} in a broader setting.  
In simple terms, \cite[Theorem~2.12]{CF1997} states that given a structured deformation $(g,G) \in SBV(\Omega; \R{d}) \times L^1(\Omega; \R{d\times N})$, there exists a sequence $u_n \in SBV(\Omega; \R{d})$ such that
\begin{equation}\label{1000}
u_n \to g \;\; \text{in}\;  L^1(\Omega; \R{d}) \qquad \text{and}\qquad \nabla u_n \wsto G\;\;\text{in}\;\cM(\Omega;\R{d\times N}).
\end{equation}
Its proof is a consequence of the following two results. 
\begin{theorem}[{\cite[Theorem~3]{AL}}]\label{Al}
Let $f \in L^1(\Omega; \R{d{\times} N})$. 
Then there exist $u \in SBV(\Omega; \R d)$, a Borel function $\beta\colon\Omega\to\R{d{\times} N}$, and a constant $C>0$ depending only on $N$, such that
\begin{equation}\label{817}
Du = f \,{\cL}^N + \beta \cH^{N-1}\res S_u, \qquad
\int_{S_u\cap\Omega} |\beta(x)| \, \de \cH^{N-1}(x) \leq C \lVert f\rVert_{L^1(\Omega; \R{d {\times} N})}.
\end{equation}
\end{theorem}
\begin{lemma}[{\cite[Lemma~2.9]{CF1997}}]\label{ctap}
Let $u \in BV(\Omega; \R d)$. Then there exist piecewise constant functions $\bar u_n\in SBV(\Omega;\R d)$  such that $\bar u_n \to u$ in $L^1(\Omega; \R d)$ and
\begin{equation}\label{818}
|Du|(\Omega) = \lim_{n\to +\infty}| D\bar u_n|(\Omega) = \lim_{n\to +\infty} \int_{S_{\bar u_n}} |[\bar u_n](x)|\; \de\cH^{N-1}(x).
\end{equation}
\end{lemma}

We now recall the notion of a multi-level structured deformation, see \cite[Section~2]{DO2019} and \cite[Section~3.2.7]{book}. 
\begin{definition}\label{Def2.1}
For $L\in\N{}\setminus\{0\}$, $p\geq 1$, and $\Omega\subset\R{N}$, we define the set of \emph{$(L+1)$-level (first-order) structured deformations} on $\Omega$ as
$$SD_{L,p}(\Omega)\coloneqq SBV(\Omega;\R{d})\times \underbrace{L^1(\Omega;\R{d\times N})\times\cdots\times L^1(\Omega;\R{d\times N})}_{(L-1)\text{-times}}\times L^p(\Omega;\mathbb R^{d\times N})$$
In particular when $p=1$, the space will be simply denoted by $SD_L(\Omega)$, i.e.
$$SD_{L}(\Omega)\coloneqq SBV(\Omega;\R{d})\times \underbrace{L^1(\Omega;\R{d\times N})\times\cdots\times L^1(\Omega;\R{d\times N})}_{L\text{-times}}. $$ 
\end{definition} 
In the case $L=1$, the space $SD_1(\Omega)$ coincides with the space $SD(\Omega)$ in \eqref{eq_SD}, introduced in \cite{CF1997}.

The convergence of a (multi-indexed) sequence of $SBV$ functions to an $(L+1)$-level structured deformation $(g,G_1,\ldots,G_L)$, belonging to  $SD_{L,p}(\Omega)$, is defined as follows. 
\begin{definition}[{\cite[Definition~3.5]{book}}]\label{S000}
Let $L\in\N{}\setminus\{0\}$ and $p\geq1$, and let $(g,G_1,\ldots,G_L)\in SD_{L,p}(\Omega)$, and let $\N{L}\ni(n_1,\ldots,n_L)\mapsto u_{n_1,\ldots,n_L}\in SBV(\Omega;\R{d})$ be a (multi-indexed) sequence. We say that $\big(u_{n_1,\ldots,n_L}\big)$ converges in the sense of $SD_{L,p}(\Omega)$ to $(g,G_1,\ldots,G_L)$ if
\begin{itemize}
	\item[(i)] $\displaystyle\lim_{n_1}\cdots\lim_{n_L} u_{n_1,\ldots,n_L} = g$, with each of the iterated limits in the sense of $L^1(\Omega;\R{d})$;
	\item[(ii)] for all $\ell=1,\ldots,L-1$, there exists a sequence $\N{\ell}\ni(n_1,\ldots,n_\ell)\mapsto g_{n_1,\ldots,n_\ell} \in SBV(\Omega;\R{d})$ such that
		$$\begin{cases}
		\displaystyle \lim_{n_{\ell+1}}\cdots\lim_{n_L} u_{n_1,\ldots,n_L} = g_{n_1,\ldots,n_\ell}\,, & \text{with each of the iterated limits in the sense of $L^1(\Omega;\R{d})$,}\\[2mm]
		\displaystyle \lim_{n_1}\cdots\lim_{n_{\ell}} \nabla g_{n_1,\ldots,n_\ell}=G_{\ell}\,, & \text{with each of the iterated limits in the sense}\\
		&\text{of weak* convergence in $\cM(\Omega;\R{d\times N})$};
		\end{cases}$$
		\item[(iii)] $\displaystyle \lim_{n_1}\cdots\lim_{n_L} \nabla u_{n_1,\ldots,n_L} = G_L$ with each of the iterated limits in the sense of weak* convergence in  $\cM(\Omega;\R{d\times N})$.
	\end{itemize}
We use the notation $u_{n_1,\ldots,n_L}\weakstH (g,G_1,\ldots,G_L)$ to indicate this convergence.
Note that, if $L=1$, condition (ii) above is void.
\end{definition}
\begin{remark}\label{remLp}
We make the following observations on Definition~\ref{S000}:
\begin{itemize}
\item Notice that, if $p>1$ and the additional condition
$$\sup_{n_1}\cdots\sup_{n_L}\int_\Omega |\nabla u_{n_1,\ldots,n_L}|^p \,\de x < +\infty$$
holds true, then the iterated limits in (iii) hold indeed in the sense of $L^p$-weak convergence.
\item In the case $L=1=p$ Definition~\ref{S000} recovers the notion of convergence of $(u_n)$ to $(g,G)$ in the sense of \cite[Theorem~2.12]{CF1997}. 
\item For the sake of clarity, we write explicitly the convergence in the case of a $3$-level structured deformation $(g,G_1,G_2)\in SD_{2,p}(\Omega)$. A double-indexed sequence $\big(u_{n_1, n_2}\big)$ converges to $(g,G_1,G_2)$ in the sense of Definition \ref{S000} provided that
\begin{itemize}
\item[(i)] $ \displaystyle \lim_{n_1}\lim_{n_2} u_{n_1, n_2} = g$
in $L^1(\Omega;\mathbb R^{d})$;
\item[(ii)] there exists a sequence $(g_{n_1}) \subset SBV(\Omega;\R{d})$ such that $\displaystyle\lim_{n_2} u_{n_1, n_2} = g_{n_1}$ in $L^1(\Omega;\R{d})$ (for every $n_1\in\N{}$), and $\displaystyle \lim_{n_1} \nabla g_{n_1} = G_1$\,, weakly* in $\cM(\Omega;\R{d\times N})$;
\item[(iii)] $ \displaystyle\lim_{n_1}\lim_{n_2}  \nabla u_{n_1, n_2} = G_2,$ weakly* in $\cM(\Omega;\mathbb R^{d\times N})$.
\end{itemize}
\end{itemize} 
\end{remark}

\subsection{Approximation theorem and auxiliary results}

The ideas behind the construction of a sequence satisfying \eqref{1000}, namely Theorem \ref{Al} and Lemma \ref{ctap}, allow us to obtain a (multi-indexed) sequence $(u_{n_1,\ldots,n_L})$ that approximates an $(L+1)$-level structured deformation $(g,G_1,\ldots,G_L)$. We thus obtain the following result, whose proof is an immediate adaptation of \cite[Theorem 3.3]{BMMOZ}.

\begin{theorem}[{Approximation Theorem for $(L+1)$-level structured deformations}]\label{appTHMh}
	Let $L\in\N{}\setminus\{0\}$ and $(g,G_1,\ldots,G_L)\in SD_{L,p}(\Omega)$. 
	Then there exists a sequence $(n_1,\ldots,n_L)\mapsto u_{n_1,\ldots,n_L}\in SBV(\Omega;\R{d})$
	converging to $(g,G_1,\ldots,G_L)$ in the sense of Definition~\ref{S000}.
\end{theorem}

The following corollary provides an alternative proof of the Approximation Theorem~\ref{appTHMh} which will be useful to prove our representation theorems (see Theorem~\ref{thm_approx_CF} and Theorem \ref{mainL} below). 
For simplicity of exposition, we state and prove it in  $SD_{2,p}(\Omega)$; the case for $L=3$ is presented in Appendix~\ref{app_L=3}. The arguments given there can be iterated for a general $L \in \mathbb N$. 

\begin{corollary}\label{S004}
Let $p\geq1$. For every $(g,G_1,G_2)\in SD_{2,p}(\Omega)$ and for every sequence of structured deformations $\big((\gamma_{n_1},\Gamma_{n_1})\big)$ in $SD_{1}(\Omega)$ such that $(\gamma_{n_1})$ converges to $(g,G_1)$ in the sense of Definition~\ref{S000} and $\Gamma_{n_1}\wsto G_2$ in $\cM(\Omega;\R{d\times N})$, there exists a sequence $(n_1,n_2)\mapsto u_{n_1,n_2}$ converging to $(g,G_1,G_2)$ in the sense of Definition~\ref{S000}.
\end{corollary}
\begin{proof}
By the Approximation Theorem \cite[Theorem~2.12]{CF1997}, for every $n_1\in\N{}$, there exists a sequence $n_2\mapsto v_{n_2}^{(n_1)}\in SBV(\Omega;\R{d})$ such that $v_{n_2}^{(n_1)}\to \gamma_{n_1}$ in $L^1(\Omega;\R{d})$  and $\nabla v_{n_2}^{(n_1)}\wsto \Gamma_{n_1}$ in $\cM(\Omega;\R{d\times N})$ as $n_2\to\infty$. In fact, the sequence $n_2\mapsto v_{n_2}^{(n_1)}$ satisfies
\begin{equation}\label{apprGamma}\nabla v_{n_2}^{(n_1)} = \Gamma_{n_1}.
	\end{equation}
Then the sequence 
\begin{equation}\label{defappr}(n_1,n_2)\mapsto u_{n_1,n_2}\coloneqq v_{n_2}^{(n_1)}
	\end{equation}
approximates $(g,G_1,G_2)$ in the sense of Definition~\ref{S000}.
Indeed, 
$$\lim_{n_1}\lim_{n_2} u_{n_1,n_2}=\lim_{n_1}\lim_{n_2} v_{n_2}^{(n_1)}=\lim_{n_1} \gamma_{n_1}=g \, \mbox{ in } L^1(\Omega;\R{d}),$$ 
which proves part (i).
Moreover, since $\gamma_{n_1}\in SBV(\Omega;\R{d})$ for all $n_1$ and
$$\lim_{n_1} \nabla \Big(\lim_{n_2} u_{n_1,n_2}\Big)=\lim_{n_1} \nabla \Big(\lim_{n_2} v_{n_2}^{(n_1)}\Big)=\lim_{n_1} \nabla \gamma_{n_1}=G_1 \,
\mbox{ weakly* in } \cM(\Omega;\R{d\times N}),$$
part (ii) is proved.
Finally,
$$\lim_{n_1} \lim_{n_2} \nabla u_{n_1,n_2}=\lim_{n_1} \lim_{n_2} \nabla v_{n_2}^{(n_1)}=\lim_{n_1} \Gamma_{n_1}=G_2 \, \mbox{ weakly* in } \cM(\Omega;\R{d\times N}),$$
which proves part (iii).
\end{proof}
\begin{remark}\label{casepRem}
\begin{itemize}
\item Notice that a sequence $\big((\gamma_{n_1},\Gamma_{n_1})\big)$ always exists: indeed, the existence of $\gamma_{n_1}$ is ensured by the Approximation Theorem \cite[Theorem~2.12]{CF1997}, whereas one can always take the constant sequence $\Gamma_{n_1}=G_2$.

\item If $p>1$, suppose that, in addition to the hypotheses of Corollary~\ref{S004}, we have that for every $n_1\in\N{}$, $(\gamma_{n_1},\Gamma_{n_1})\in SD_{1,p}(\Omega)$ and  $\Gamma_{n_1}\wto G_2$ weakly in $L^p(\Omega;\R{d\times N})$.
Then the constructed sequence $(n_1,n_2)\mapsto u_{n_1,n_2}$ satisfies condition (iii) of Definition~\ref{S000} weakly in $L^p(\Omega;\R{d\times N})$.
\end{itemize}
\end{remark}

In order to obtain the relaxation results in Section \ref{multi}, we state a variant of the global method for relaxation, suitable for the space $SD_{L,p}(\Omega)$,  that was obtained in \cite{BMZ2024}. We point out that the current version does not differ from \cite[Theorem 3.2]{BMZ2024} in the cases $L=1$ and/or $p=1$, since, in this setting, the function space $HSD^p_L(\Omega)$ therein coincides with $SD_{L,p}(\Omega)$.

Let $\mathcal F\colon SD_{L,p}(\Omega) \times\mathcal 
O(\Omega)\to [0, +\infty]$ be a functional satisfying the following hypotheses: 
\begin{enumerate}
	\item[(H1)] for every $(g, G_1, \dots, G_L) \in SD_{L,p}(\Omega)$, $\mathcal F(g, G_1,\dots, G_L;\cdot)$ is the restriction to 
	$\mathcal O(\Omega)$ of a Radon measure; 
	\item[(H2)] for every $O \in \mathcal O(\Omega)$,  
	$\mathcal F(\cdot, \cdot, \dots; O)$ is lower semicontinuous, in the following sense:
	if $(g, G_1,\dots, G_L) \in SD_{L,p}(\Omega)$ and $((g^n, G_1^n,\dots, G^n_L)) \subset SD_{L,p}(\Omega)$ are such that $g^n \to g$ in $L^1(\Omega;\mathbb R^{d})$ and  $G_i^n \wsto G_i$ in $\cM(\Omega;\R{d\times N})$, for $i=1,\dots, L$, then
	$$
	\mathcal F(g, G_1,\dots, G_L;O)\leq \liminf_{n}\mathcal F(g^n,G_1^n,\dots, G^n_L;O);
	$$
	\item[(H3)] for all $O \in \mathcal O(\Omega)$, $\mathcal F(\cdot, \ldots, \cdot;O)$ is local, that is, if $g= u$, $G_1= U_1$, $\dots$, $G_L=U_L$ a.e. in $O$, then 
	$\mathcal F(g, G_1,\dots, G_L;O)= \mathcal F(u, U_1,\dots, U_L;O)$;
	\item[(H4)] there exists a constant $C>0$ such that
	\begin{align*}
		&\frac{1}{C}\bigg(\sum_{\ell=1}^{L-1}\|G_\ell\|_{L^1(O;\mathbb R^{d \times N})} + 
			\|G_L\|^p_{L^p(O;\mathbb R^{d \times N})}+ |D g|(O)\bigg)\leq \cF(g, G_1,\dots, G_L;O)\\
		&\hspace{2cm}\leq
		C\left(\cL^N(O) + \sum_{\ell=1}^{L-1} \|G_\ell\|_{L^1(O;\mathbb R^{d \times N})} + \|G_L\|^p_{L^p(O;\mathbb R^{d \times N}) }+  |D g|(O)\right),
	\end{align*}
	for every $(g, G_1, \dots, G_L) \in SD_{L,p}(\Omega)$ and every $O \in \cO(\Omega)$.
\end{enumerate}

Given $(g, G_1, \dots, G_L)\in SD_{L,p}(\Omega)$ and 
$O \in \mathcal O(\Omega)$, we introduce the space of test functions 
\begin{align}
	\cC_{SD_{L,p}}(g, G_1, \dots, G_L; O)\coloneqq \left\{(u, U_1, \dots, U_L)\in SD_{L,p}(\Omega): u=g \hbox{ in a neighbourhood of } \partial O, \right. \nonumber\\
	\left.\int_O (G_i-U_i) \,\de x=0, i=1,\dots, L  \right\}, 	\label{classM}
\end{align}
and we let $m_{L,p} \colon SD_{L,p}(\Omega)\times \cO(\Omega)\to[0,+\infty)$ be the functional defined by
\begin{align}\label{mLdef}
	m_{L,p}(g, G_1,\dots, G_L;O):=\inf\Big\{\mathcal F(u, U_1,\dots, U_L; O): (u, U_1,\dots, U_L)\in \mathcal C_{SD_{L,p}}(g, G_1, \dots, G_L; O)\Big\}.
\end{align}

Then, the following result holds.
\begin{theorem}\label{globalmethod-new}
	Let $p \geq 1$ and let $\cF\colon SD_{L,p}(\Omega) \times\mathcal 
	O(\Omega)\to [0, +\infty]$ be a functional satisfying (H1)-(H4). 
	Then
	\begin{align*}
		\mathcal F(u, U_1,\dots, U_L;O)&= \int_O \!f_{L,p}(x,u(x), \nabla u(x), U_1(x),\dots, U_L(x))\,\de x \\
		&+\int_{O \cap S_u}\!\!\!\!\!\Phi_{L,p}(x, u^+(x), u^-(x),\nu_u(x)) 
		\,\de \mathcal H^{N-1}(x),\end{align*}
	where
	\begin{align}\label{fL}
		f_{L,p}(x_0, \alpha, \xi, B_1,\dots, B_L):=	\limsup_{\varepsilon \to 0^+}
		\frac{m_{L,p}(\alpha + a_{\xi,x_0}, B_1,\dots, B_L; Q(x_0,\varepsilon))}{\varepsilon^N},
	\end{align}
	\begin{align}\label{PhiL}
		\Phi_{L,p}(x_0, \lambda, \theta, \nu):= \limsup_{\varepsilon \to 0^+}\frac{m_{L,p}(s_{\lambda, \theta,\nu}(\cdot-x_0), 0,\dots, 0; Q_{\nu}(x_0,\varepsilon))}{\varepsilon ^{N-1}},
	\end{align}
	for a.e. $x_0\in \Omega$, for every $\alpha, \lambda, \theta \in \mathbb R^d$, $\xi, B_1,\dots, B_L \in \mathbb R^{d \times N}$, $\nu \in \mathbb S^{N-1}$,
	where $0$ is the zero matrix in $\mathbb R^{d \times N}$, $
		a_{\xi,x_0}(x)\coloneqq \xi(x - x_0)$ 
	 and 
	$s_{\lambda,\theta, \nu}$ is defined by \begin{equation}\label{snu}
		s_{\lambda,\theta,\nu}(x)\coloneqq 
		\begin{cases}
			\lambda & \text{if $x\cdot\nu\geq0$,} \\
			\theta & \text{if $x\cdot\nu<0$.}
		\end{cases}
	\end{equation}
\end{theorem}
\begin{proof} The proof is identical to the one in \cite[Theorem 3.2]{BMZ2024}, upon noticing that the fields $G_i, i = 1, \ldots L$ behave independently of each other.
\end{proof}

\begin{remark}\label{remtrasl}
	As in \cite[Remark 3.3]{BMZ2024}, it follows that if $\mathcal F$ is translation invariant in the first variable, i.e. if
	$$\mathcal F(u+b, U_1, \dots, U_L;O)= \mathcal F(u, U_1,\dots, U_L;O),$$ 
	for every $((u,U_1,\dots, U_l),O) \in SD_{L,p}(\Omega)\times \mathcal O(\Omega)$ and for every $b\in \mathbb R^d$,
	then the function $f_{L,p}$ in \eqref{fL} does not depend on $\alpha$ and the function $\Phi_{L,p}$ in \eqref{PhiL} does not depend separately on $\lambda$ and $\theta$ but only on the difference $\lambda- \theta$.
 With an abuse of notation
	we write
	$$f_{L,p}(x_0,\xi, B_1,\dots, B_L) = f_{L,p}(x_0, \alpha, \xi, B_1,\dots, B_L)  \quad \mbox{and}
	\quad \Phi_{L,p}(x_0, \lambda - \theta, \nu) = \Phi_{L,p}(x_0, \lambda, \theta, \nu).$$ 
\end{remark}

\section{A generalization of the relaxation theorem for $2$-level structured deformations}\label{sec_gen_CF}
In this section, we generalize some relaxation results already available for $2$-level structured deformations (\cite[Theorems~2.16 and~2.17]{CF1997}, \cite[Theorem~5.1]{MMOZ}, \cite[Theorems~4.1 and~4.2]{BMZ2024}) by weakening the hypotheses on the initial energy densities $W$ and $\psi$.

\begin{assumptions}[{see \cite[Definition~2.5]{BMMOZ}}]\label{gen_ass}
For $q\geq1$, we denote by $\cE\cD(q)$ the collection of pairs $(W,\psi)$ of bulk and surface energy densities $W\colon \Omega\times\R{d\times N} \to [0, +\infty)$ and $\psi\colon \Omega\times\R{d}\times \S{N-1} \to [0, +\infty)$ satisfying the following conditions 
\begin{enumerate}
\item \label{W_cara} $W$ is a Carath\'{e}odory function, namely it is measurable with respect to $x\in\Omega$ and continuous with respect to $A\in\R{d\times N}$;
\item \label{(W1)_p} there exists $C_W>0$ such that, for a.e. $x\in\Omega$ and $A_1,A_2 \in \R{d\times N}$,
\begin{equation*}
|W(x,A_1) - W(x,A_2)| \leq C_W|A_1 - A_2| \big(1+|A_1|^{q-1}+|A_2|^{q-1}\big);
\end{equation*}
\item \label{W1bdd}   there exists $A_0 \in \mathbb R^{d \times N}$ such that 
$W(\cdot, A_0)\in L^\infty(\Omega)$;
\color{black}
\item \label{Wcoerc} there exists $c_W >0$ such that
$$
W(x,A) \geq c_W |A|^q- \frac{1}{c_W}
$$
for a.e. $x \in \Omega$ and every $A \in \mathbb R^{d \times N}$.
\item \label{(psi-1)} $\psi$ is continuous with respect to all of its variables; 
\item \label{(psi0)} (symmetry) for every $x \in \Omega, \lambda \in \mathbb R^d, \nu \in \mathbb S^{N-1},$
$$
\psi(x,\lambda, \nu)=\psi(x, -\lambda, - \nu);
$$
\item \label{(psi1)} there exists $c_\psi,C_\psi > 0$ such that, for all $x\in\Omega$, $\lambda \in \R{d}$, and $\nu \in \S{N-1}$,
$$c_\psi|\lambda| \leq \psi(x,\lambda, \nu) \leq C_\psi|\lambda |;$$
\item \label{(psi2)} (positive $1$-homogeneity) for all $x\in\Omega$, $\lambda \in \R{d}$, $\nu \in \S{N-1}$, and $t >0$
$$\psi(x,t\lambda, \nu) = t\psi(x, \lambda, \nu);$$
\item \label{(psi3)} (sub-additivity) for all $x\in\Omega$, $\lambda_1,\lambda_2 \in \R{d}$, and $\nu \in \S{N-1}$,
\begin{equation*}
\psi(x, \lambda_1 + \lambda_2, \nu) \leq \psi(x,\lambda_1, \nu) +\psi(x,\lambda_2, \nu);
\end{equation*}
\item \label{(psi4)} there exists a continuous function $\omega_\psi\colon[0,+\infty)\to[0,+\infty)$ with $\omega_\psi(s)\to 0$ as $s\to0^+$ such that, for every $x_0,x\in\Omega$, $\lambda \in \R{d}$, and $\nu \in \S{N-1}$,
\begin{equation*}
|\psi(x,\lambda,\nu)-\psi(x_0,\lambda,\nu)|\leq\omega_\psi(|x-x_0|)|\lambda|.
\end{equation*}
\end{enumerate}
\end{assumptions}

Notice that conditions  \eqref{(psi1)} and \eqref{(psi3)} imply that~$\psi$ is Lipschitz continuous in the second variable, namely, for all $x\in\Omega$, $\lambda_1,\lambda_2\in\R{d}$, and $\nu\in\mathbb{S}^{N-1}$
\begin{equation}\label{psi_Lip}
|\psi(x,\lambda_1,\nu)-\psi(x,\lambda_2,\nu)|\leq C_\psi|\lambda_1-\lambda_2|
\end{equation}
(where $C_\psi$ is the same constant of property \eqref{(psi1)}).

For $(W,\psi)\in \cE\cD(p)$, for some $p\geq1$, we consider the initial energy of a deformation $u\in SBV(\Omega;\mathbb R^d)$ defined by
\begin{equation}\label{103}
E(u)\coloneqq \int_\Omega W(x,\nabla u(x))\,\de x+\int_{\Omega\cap S_u} \psi(x,[u](x),\nu_u(x))\,\de\cH^{N-1}(x),
\end{equation}
and, following \cite{CF1997}, we assign an energy to a ($2$-level) structured deformation $(g,G)\in SD_{1,p}(\Omega)$, via
\begin{equation}\label{102}
	I_{1,p}(g,G)\coloneqq \inf\Big\{\liminf_{n\to\infty} E(u_n): u_n \in SBV(\Omega;\R{d}), u_n\weakstH(g,G)\Big\}.
\end{equation}

We introduce some notations which will be used throughout the discussion.
For $x_0 \in \Omega$ and $A\in\R{d\times N}$, we define the affine and linear functions of gradient~$A$ by
\begin{equation}\label{affine}
a_{A,x_0}(x)\coloneqq A(x - x_0)\qquad\text{and}\qquad \ell_A(x)\coloneqq Ax, 
\end{equation} 
respectively.
For $p\geq1$ and $A,B\in\R{d\times N}$, we let
\begin{equation}\label{T001}
	\cC^{\bulk}_p(A,B)\coloneqq \bigg\{u\in SBV(Q;\R{d}) :  u=\ell_A 
	\hbox{ in a neighborhood of } \partial Q, 
	 \int_Q \nabla u\,\de x=B, |\nabla u|\in L^p(Q) \bigg\},
\end{equation}
and, for $\lambda\in\R{d}$ and $\nu\in\S{N-1}$, we let
\begin{equation}\label{T002}
	\cC^\surface(\lambda,\nu)\coloneqq \bigg\{u\in SBV(Q_\nu;\R{d}): u=s_{\lambda,0,\nu} \hbox{ in a neighborhood of }\partial Q_{\nu}\,, \int_{Q_\nu}\nabla u\,\de x=0	\bigg\},
\end{equation}
where, the function $s_{\lambda,0, \nu}$ is defined in \eqref{snu}.
For $\eps>0$ and $O\in\mathcal{O}(\Omega)$, we define the following localizations of the energy~$E$ in~\eqref{103} 
\begin{subequations}\label{Erelax}
\begin{eqnarray}
E_{x,\eps}^\surface(u;O) & \!\!\!\!\coloneqq & \!\!\!\!\!\! \int_{S_u\cap O} \psi(x+\eps y,[u](y),\nu_u(y))\,\de\cH^{N-1}(y), \label{Erelaxs} \\
E_{x,\eps}^{\bulk}(u;O) & \!\!\!\!\coloneqq & \!\!\!\!\!\! \int_{O} W(x+\eps y,\nabla u(y))\,\de y + E_{x,\eps}^\surface(u;O), \label{Erelaxb}\\
\widetilde{E}_{x,\eps}(u;O) &\!\!\!\!\coloneqq &\!\!\!\!  \int_{O} \eps\, W(x+\eps y,\eps^{-1}\nabla u(y))\,\de y + E_{x,\eps}^\surface(u;O);\label{Erelaxrec}
	\end{eqnarray}
\end{subequations}
for a.e.~$x\in\Omega$ and $A,B\in\R{d\times N}$, we define
\begin{equation}\label{906}
	H_{1,p}(x,A,B)\coloneqq \limsup_{\eps\to0}\inf\Big\{E_{x,\eps}^{\bulk}(u;Q): u\in\cC_p^\bulk(A,B)\Big\},
\end{equation}
and for a.e.~$x\in\Omega$, $\lambda\in\R{d}$, and $\nu\in\S{N-1}$, we define 
\begin{equation}\label{907}
	h_{1,p}
	(x,\lambda,\nu)\coloneqq \limsup_{\eps\to0} \inf\Big\{ \delta_1(p) \widetilde{E}_{x,\eps}(u;Q_\nu)+(1-\delta_1(p)) 
	E_{x,\eps}^\surface(u;Q_\nu): u\in\cC
	^\surface(\lambda,\nu)\Big\}.
\end{equation}
\begin{remark}\label{rem3.2}
We comment on the available literature concerning the energy $I_{1,p}$ defined in~\eqref{102}.
\begin{itemize}
\item[(a)] In the case of~$W$ and~$\psi$ independent of the $x$-variable, an integral representation result for the relaxed energy $I_{1,p}$ defined in~\eqref{102} was proved in \cite[Theorems~2.16 and~2.17]{CF1997}, where the relaxed bulk and surface energy densities $(A,B)\mapsto H_{1,p}(A,B)$ and $(\lambda,\nu)\mapsto h_{1,p}(\lambda,\nu)$ are obtained by solving the infimization problems in \eqref{906} and \eqref{907} (with the obvious modifications), namely
$$I_{1,p}(g,G)=\int_\Omega H_{1,p}(\nabla g(x),G(x))\,\de x+\int_{\Omega\cap S_g} h_{1,p}([g](x),\nu_g(x))\,\de\cH^{N-1}(x),$$ 
compare with \cite[Theorem~2.16 and formula (2.15)]{CF1997}\footnote{In \cite[formula (2.15)]{CF1997}, the dependence on the normal in the relaxed surface energy density (see the formula for~$h$, \cite[formula (2.17)]{CF1997}) was mistakenly omitted, as already noted in \cite[Theorem 3]{OP2015} and \cite[formula (4)]{S17}.}.

\item[(b)] In the case of $W$ continuous in the~$x$ variable with modulus of continuity $\omega_W$, an integral representation result for the relaxed energy $I_{1,p}$ defined in~\eqref{102} is presented in \cite[Theorem~3.2]{book}, where the relaxed bulk and energy densities~$H_{1,p}$ and~$h_{1,p}$ are obtained by solving the infimization problems in \eqref{906} and \eqref{907} for $\eps=0$, relying on techniques from~\cite{BBBF1996} to deal with the $x$-dependence. 
The integral representation reads
$$I_{1,p}(g,G)=\int_\Omega H_{1,p}(x,\nabla g(x),G(x))\,\de x+\int_{\Omega\cap S_g} h_{1,p}(x,[g](x),\nu_g(x))\,\de\cH^{N-1}(x);$$ 
compare with \cite[formula (3.16)]{book}.

\item[(c)] In the present case, we will resort to the global method for relaxation of~\cite{BFM1998} because of the weaker assumptions on~$W$, extending its formulation to the context of hierarchical structured deformations considered in~\cite{BMZ2024}.
\end{itemize}

We comment on formulas~\eqref{Erelax} and on the minimization problems~\eqref{906} and~\eqref{907}.
\begin{itemize}
\item[(d)] We notice that, thanks to~\eqref{(psi4)} in Assumptions~\ref{gen_ass}, the dependence of $E_{x,\eps}^{\surface}$ on $\eps$ can be dropped: by using the techniques in~\cite{BBBF1996}, it is easy to see one can replace $E_{x,\eps}^{\surface}$ by $E_{x,0}^{\surface}\eqqcolon E_x^{\surface}$ in \eqref{907}, so that, since the integration variable in~\eqref{Erelaxs} is~$y$, the first entry of $\psi$ can be frozen and therefore considered illusory for the minimization process; hence, formulas~\eqref{Erelax} become
\begin{subequations}\label{Erelax1}
\begin{eqnarray}
E_{x}^\surface(u;O) & \!\!\!\!\coloneqq & \!\!\!\!\!\! \int_{S_u\cap O} \psi(x,[u](y),\nu_u(y))\,\de\cH^{N-1}(y), \label{Erelax1s} \\
E_{x,\eps}^{\bulk}(u;O) & \!\!\!\!\coloneqq & \!\!\!\!\!\! \int_{O} W(x+\eps y,\nabla u(y))\,\de y + E_{x}^\surface(u;O),\label{Erelax1b}\\
\widetilde{E}_{x,\eps}(u;O) &\!\!\!\!\coloneqq &\!\!\!\!  \int_{O} \eps\, W(x+\eps y,\eps^{-1}\nabla u(y))\,\de y + E_{x}^\surface(u;O).\label{Erelax1rec}
	\end{eqnarray}
\end{subequations}

\item[(e)] In view of (d) above, we notice that, if $p>1$, the minimization problem in \eqref{907} reduces to 
$$h_{1,p}	(x,\lambda,\nu)= \limsup_{\eps\to0} \inf\big\{E_{x,\eps}^\surface(u;Q_\nu): u\in\cC^\surface(\lambda,\nu)\big\} = \inf\big\{E_{x}^\surface(u;Q_\nu): u\in\cC^\surface(\lambda,\nu)\big\}.$$
This enlarges the class originally considered in \cite[formula (2.17)]{CF1997}, where the class $\cC^{\surface}(\lambda,\nu)$ was defined with the condition $\nabla u=0$ a.e.~in $Q_{\nu}$ (instead of $\int_{Q_\nu} \nabla u\,\de x=0$). 
The equivalence of the minimization problems in these two different classes was obtained in \cite{nogap} and \cite{S17} as a byproduct of relaxation, whereas it was proved as a property of the energy in~\cite{EKM} for a general class of $x$-independent surface energy densities. 
Recalling again (d) above, this is meaningful in our context because our initial surface energy densities $\psi$ can be considered independent of $x$, owing to~\eqref{(psi4)} in Assumptions~\ref{gen_ass}.

\item[(f)] In the case of $W$ continuous in the~$x$ variable with modulus of continuity $\omega_W$, also formula~\eqref{Erelax1b} has no $\eps$-dependence, and reads
\begin{equation}\label{Erelax_b_good}
E_{x}^{\bulk}(u;O) =\int_{O} W(x,\nabla u(y))\,\de y + E_{x}^\surface(u;O);
\end{equation}
moreover, if there exist $0<\alpha <1$, $C>0$ such that
$$
\left|W^\infty(x,A)- \frac{W(x,tA)}{t}\right|\leq \frac{C |A|^{1-\alpha}}{t^\alpha}\quad\text{for every $x \in \Omega$, for every $t>0$ such that $|tA|\geq 1$,}
$$ 
where $W^\infty$ is the recession function of $W$, defined by $W^\infty(x,A)\coloneqq \limsup_{t\to+\infty} t^{-1}W(x,tA)$, then the localized energy $\widetilde{E}_{x,\eps}$ in~\eqref{Erelax1rec} becomes
\begin{equation}\label{Erelaxrec_really}
E_{x}^\rec(u;O)\coloneqq \int_{O}  W^\infty(x,\nabla u(y))\,\de y + E_{x}^\surface(u;O).
\end{equation}
\item[(g)] In view of point (f), formulas~\eqref{906} and~\eqref{907} reduce to
\begin{equation*}
\begin{split}
H_{1,p}(x,A,B) =&\, \inf\Big\{E_{x}^{\bulk}(u;Q): u\in\cC_p^\bulk(A,B)\Big\},\\
h_{1,p}(x,\lambda,\nu) =&\, \inf\Big\{ \delta_1(p) E_{x}^{\rec}(u;Q_\nu)+(1-\delta_1(p)) E_{x}^\surface(u;Q_\nu): u\in\cC^\surface(\lambda,\nu)\Big\},
\end{split}
\end{equation*}
which are those in~\cite[formulas (3.15)]{book}.
\end{itemize}
\end{remark}

\smallskip

Given $(g, G)\in SD_{1,p}(\Omega)$ and 
$O \in \cO(\Omega)$, we introduce the class (of competitors) 
\begin{align}\label{testf}
	\mathcal C_{SD_{1,p}}(g, G; O) \coloneqq \bigg\{(u, U)\in SD_{1,p} (\Omega): u=g \hbox{ in a neighbourhood of } \partial O, 	\int_O (G-U) \,\de x=0 \bigg\},
\end{align}
and, with a slight abuse of notation, we let  
$m_{1,p}\colon SD_{1,p}(\Omega)\times \cO(\Omega)\to[0,+\infty)$ be the functional defined by
\begin{align}
	m_{1,p}(g, G;O) \coloneqq \inf\Big\{ I_{1,p}(u, U; O): (u, U)\in \mathcal C_{SD_{1,p}}(g, G; O)\Big\},
\end{align}
where $I_{1,p}(\cdot,\cdot;O)$ is the localization of the functional in~\eqref{102} to $O\in\cO(\Omega)$, i.e. $m_{1,p}$ is the functional defined in \eqref{mLdef}, for $L=1$ and associated to $\mathcal F=I_{1,p}$.

The following integral representation theorem is a consequence of \cite[Theorem 3.2]{BMZ2024}.
\begin{theorem}\label{thm_approx_CF}
Let $p\geq1$ and let $(W,\psi)\in\cE\cD(p)$; let $(g,G)\in SD_{1,p}(\Omega)$ and let $I_{1,p}(g,G)$ be defined by~\eqref{102}. 
Then $I_{1,p}$ admits the integral representation 
\begin{equation}\label{int_rep_I}
I_{1,p}(g,G)=\int_\Omega H_{1,p}(x,\nabla g(x), G(x))\,\de x+\int_{\Omega\cap S_g} h_{1,p}(x,[g](x),\nu_g(x))\,\de \cH^{N-1}(x),
\end{equation}
where $H_{1,p}$ and $h_{1,p}$ are defined in~\eqref{906} and~\eqref{907}, respectively.
\end{theorem}
\begin{proof}
By \cite[Theorem 4.1]{BMZ2024}, the localized version $\mathcal{O}(\Omega)\ni O\mapsto I_{1,p}(g,G;O)$ of the functional $I_{1,p}$ satisfies all the properties in \cite[Theorem 3.2]{BMZ2024}. In particular, it is the restriction to $\mathcal{O}(\Omega)$ of a Radon measure which is absolutely continuous with respect to $\mathcal L^N+ |Dg|+ |G|$ and it admits the following representation
\begin{align*}
I_{1,p}(u, U;O)= \int_O  f_{1,p}(x, u(x), \nabla u(x), U(x))\,\de x +\int_{O \cap S_u} \Phi_{1,p}(x, u^+(x), u^-(x),\nu_u(x)) 
	\,\de \mathcal H^{N-1}(x),\end{align*}
where, for all $x_0\in \Omega$, $\alpha, \theta,\lambda \in \mathbb R^d$, $A, B\in \mathbb R^{d \times N}$, and $\nu \in \mathbb S^{N-1}$, 
\begin{subequations}
\begin{eqnarray}
f_{1,p}(x_0,\alpha, A, B) &\!\!\!\!\coloneqq&\!\!\!\! \limsup_{\varepsilon \to 0}\frac{m_{1,p}(\alpha + a_{A,x_0}, B; Q(x_0,\varepsilon))}{\varepsilon^N}, \\ 
\Phi_{1,p}(x_0, \lambda, \theta, \nu) &\!\!\!\!\coloneqq&\!\!\!\! \limsup_{\varepsilon \to 0}\frac{m_{1,p}(s_{\lambda, \theta,\nu}(\cdot-x_0), 0; Q_{\nu}(x_0,\varepsilon))}{\varepsilon ^{N-1}}, \label{Phi_temp}
\end{eqnarray}
\end{subequations}
where $a_{A,x_0}$ and $s_{\lambda,\theta,\nu}$ are defined in \eqref{affine} and \eqref{snu}, respectively, and $0$ in~\eqref{Phi_temp} is the zero matrix in $\mathbb R^{d \times N}$. 

It is easy to deduce (see, e.g., \cite[Remark~3.3 and proof of Theorem 4.1]{BMZ2024}) that the functional $I_{1,p}$ defined in \eqref{102} is translation invariant in the first variable, which entails that it can be represented as 
\begin{align*}
I_{1,p}(u, U;O)= \int_O  f_{1,p}(x, \nabla u(x), U(x))\,\de x +\int_{O \cap S_u} \Phi_{1,p}(x, [u](x),\nu_u(x)) \,\de \mathcal H^{N-1}(x),
\end{align*}
where (with a slight abuse of notation)
\begin{subequations}
\begin{eqnarray}
f_{1,p}(x_0, A, B) &\!\!\!\!\coloneqq&\!\!\!\! \limsup_{\varepsilon \to 0}\frac{m_{1,p}( a_{A,x_0}, B; Q(x_0,\varepsilon))}{\varepsilon^N}, \label{f3.2}\\
\Phi_{1,p}(x_0, \lambda, \nu) &\!\!\!\!\coloneqq&\!\!\!\! \limsup_{\varepsilon \to 0}\frac{m_{1,p}(s_{\lambda, 0,\nu}(\cdot-x_0), 0; Q_{\nu}(x_0,\varepsilon))}{\varepsilon ^{N-1}}. \label{Phi}
\end{eqnarray}
\end{subequations}

We now need to prove that, for all $x_0\in\Omega$, $A,B\in\R{d\times N}$, $\lambda\in\R{d}$, and $\nu\in\mathbb{S}^{N-1}$, there holds $f_{1,p}(x_0,A,B)=H_{1,p}(x_0,A,B)$ and that $\Phi_{1,p}(x_0,\lambda,\nu)=h_{1,p}(x_0,\lambda,\nu)$.

\smallskip

\noindent\emph{Step 1 -- the bulk energy density.} We claim that for a.~e.~$x_0\in\Omega$ and for every $A,B\in\R{d\times N}$, 
\begin{equation}\label{clean}
f_{1,p}(x_0,A,B)=H_{1,p}(x_0,A,B)=\widetilde{H}_{1,p}(x_0,A,B),
\end{equation}
where 
\begin{align*}
\widetilde{H}_{1,p}(x_0,A,B) \coloneqq& \limsup_{\varepsilon \to 0} \inf\Big\{\liminf_{n\to \infty}
E^\bulk_{x_0,\eps}(u_n;Q) : \{u_n\} \subset SBV(Q;\mathbb R^d), u_n \to \ell_A  \hbox{ in } L^1(Q;\mathbb R^d),\\
&\phantom{=\limsup_{\varepsilon \to 0^+} \inf\Big\{\liminf_{n\to \infty}
E^\bulk_{x_0,\eps}(u_n;Q):} 
\nabla u_n \wsto B \mbox{ in } \cM(Q;\mathbb R^{d\times N})\Big\}
\end{align*}
(notice that, if $p>1$, in view of the fact that $(W,\psi)\in\ED(p)$, the convergence $\nabla u_n\wsto B$ is improved to $\nabla u_n\wto B$ weakly in $L^p(\Omega;\R{d\times N})$).

The proof that $H_{1,p}(x_0,A,B) \leq \widetilde{H}_{1,p}(x_0,A,B)$ can be obtained as in \cite[Step 2 in Proposition 3.1]{CF1997}, first fixing $\varepsilon$ and then letting $\varepsilon \to 0$.  A careful inspection of the proofs of \cite[Lemma 2.21 and Step 2 of Proposition 3.1]{CF1997} shows both that the $x$-dependence is not an issue and that the admissible functions used in the definition of the bulk energy density in \cite[equation (2.16)]{CF1997}) belong to $\cC^\bulk_p(A,B)$.

Moreover, by \cite[Theorems 3.6 and  4.1]{BMZ2024}, the relaxed bulk energy density $f_{1,p}$ in \eqref{f3.2}, is given by
\begin{equation}\label{fHtilde}
f_{1,p}(x_0,A,B)= \limsup_{\varepsilon \to 0} \frac{m_{1,p}(a_{A,x_0}, B;Q(x_0,\varepsilon))}{\varepsilon^N} = \limsup_{\varepsilon \to 0}  \frac{I_{1,p}(a_{A,x_0}, B;Q(x_0,\varepsilon))}{\varepsilon^N} 
\end{equation}
for a.e.~$x_0 \in \Omega$ and every $A, B \in \mathbb R^{d \times N}$.
By a simple change of variables argument, invoking property \eqref{(psi2)}, it follows that, for a.e.~$x_0\in \Omega$ and every $A,B \in \mathbb R^{d \times N}$,
\begin{equation}\label{Htilde}
\limsup_{\varepsilon \to 0} \frac{I_{1,p}(a_{A,x_0}, B;Q(x_0,\varepsilon))}{\varepsilon^N} = \widetilde{H}_{1,p}(x_0,A,B).
\end{equation}

Now define $\widehat{m}_{1,p}^{\bulk}(A,B;Q(x_0,\varepsilon))\coloneqq \inf\big\{E(u;Q(x_0,\varepsilon)): \varepsilon^{-1}u(x_0 + \varepsilon \cdot) \in \mathcal C^\bulk_p(A,B)\big\}$.
It is easy to verify that for a.e.~$x_0\in \Omega$ and every $A,B \in \mathbb R^{d\times N}$ the following inequality holds true
$$m_{1,p}(a_{A,x_0},B;Q(x_0,\varepsilon))\leq \widehat{m}_{1,p}^{\bulk}(A,B;Q(x_0,\varepsilon));
$$
indeed, for any function $u$ that is admissible for $\widehat{m}_{1,p}^{\bulk}(A,B;Q(x_0,\varepsilon))$, 
the pair $(u,\nabla u)$ is a competitor for $m_{1,p}(a_{A,x_0},B;Q(x_0,\varepsilon))$.
Consequently, by \eqref{Htilde} and \eqref{fHtilde},
\begin{align*}
{\widetilde H}_{1,p}(x_0,A,B) &= \limsup_{\varepsilon \to 0}  \frac{I_{1,p}(a_{A,x_0}, B;Q(x_0,\varepsilon))}{\varepsilon^N} =\limsup_{\varepsilon \to 0} \frac{m_{1,p}(a_{A,x_0}, B;Q(x_0,\varepsilon))}{\varepsilon^N}\\
&\leq \limsup_{\varepsilon \to 0} \frac{\widehat{m}_{1,p}^{\bulk}(A, B;Q(x_0,\varepsilon))}{\varepsilon^N}=H_{1,p}(x_0,A,B),
\end{align*}
where in the last equality we used again a change of variables argument; therefore~\eqref{clean} is proved.

\smallskip

\noindent\emph{Step 2 -- the surface energy density.}
The equality $\Phi_{1,p}(x_0,\lambda,\nu)=h_{1,p}(x_0,\lambda,\nu)$ for $p>1$ is already contained in the proof of \cite[Theorem 4.1]{BMZ2024}, so that we only have to prove it for the case $p=1$.
We claim that for a.~e.~$x_0\in\Omega$ and for every $\lambda\in\R{d}$ and $\nu\in\mathbb{S}^{N-1}$, 
\begin{equation}\label{clean2}
\Phi_{1,1}(x_0,\lambda,\nu)=h_{1,1}(x_0,A,B)=\tilde{h}_{1,1}(x_0,A,B),
\end{equation}
where  
\begin{align*}
\tilde{h}_{1,1}(x_0,A,B) \coloneqq& \limsup_{\varepsilon \to 0} \inf\Big\{\liminf_{n\to \infty}
\widetilde{E}_{x_0,\eps}(u_n;Q_\nu) : 
\{u_n\} \subset SBV(Q_\nu;\mathbb R^d), u_n \to s_{\lambda,0,\nu}  \hbox{ in } L^1(Q_\nu;\mathbb R^d),\\
&\phantom{=\limsup_{\varepsilon \to 0^+} \inf\Big\{\liminf_{n\to \infty}
E^\bulk_{x_0,\eps}(u_n;Q):} 
\nabla u_n \wsto 0 \mbox{ in } \cM(Q_\nu;\mathbb R^{d\times N})\Big\}.
\end{align*}
The inequality $h_{1,1}(x_0,\lambda,\nu)\leq \tilde{h}_{1,1}(x_0,\lambda,\nu)$ is obtained as in Step 1 above.
For every $x_0\in \Omega$, $\lambda \in \mathbb R^{d}$, and $\nu \in \mathbb S^{N-1}$, we have
\begin{eqnarray}
&&\!\!\!\! \Phi_{1,1}(x_0,\lambda,\nu) = 
\limsup_{\varepsilon \to 0}\frac{m_{1,1}(s_{\lambda,0,\nu}(\cdot-x_0), 0; Q_{\nu}(x_0,\varepsilon))}{\varepsilon ^{N-1}}= \limsup_{\varepsilon \to 0}\frac{I_{1,1}(s_{\lambda,0,\nu}(\cdot-x_0), 0; Q_{\nu}(x_0,\varepsilon))}{\varepsilon ^{N-1}} \nonumber\\
&\!\!\!\!=&\!\!\!\! 
\limsup_{\varepsilon \to 0}\frac{1}{\varepsilon^{N-1}}
\inf\Bigg\{\liminf_{n\to\infty}\bigg[\int_{Q_\nu(x_0,\varepsilon)}  W(x,\nabla u_n(x)) \, \de x + \int_{Q_\nu(x_0,\varepsilon)\cap S_{u_n}} \psi(x, [u_n](x),\nu_{u_n}(x)) \,\de\mathcal H^{N-1}(x)\bigg] : \nonumber\\
&\!\!\!\!&\!\!\!\!
\phantom{\limsup_{\varepsilon \to 0}\frac{1}{\varepsilon^{N-1}} \inf\Bigg\{} \text{$u_n \in SBV(Q_\nu(x_0,\varepsilon);\mathbb R^d)$,  $u_n \to s_{\lambda,0,\nu}(\cdot-x_0)$ in $L^1(Q_\nu(x_0,\varepsilon);\mathbb R^d)$,} \nonumber\\
&\!\!\!\!&\!\!\!\!
\phantom{\limsup_{\varepsilon \to 0}\frac{1}{\varepsilon^{N-1}} \inf\Bigg\{} \text{$\nabla u_n \wsto 0$ in $\cM(Q_\nu(x_0,\varepsilon);\mathbb R^{d\times N})$} \Bigg\} \nonumber\\
&\!\!\!\!=&\!\!\!\! 
\limsup_{\varepsilon \to 0}
\inf\Bigg\{\liminf_{n\to\infty}\bigg[ \eps \int_{Q_\nu}  W(x_0+\eps y,\nabla u_n(x_0+\eps y)) \, \de y \nonumber\\
&\!\!\!\!&\!\!\!\!
\phantom{\limsup_{\varepsilon \to 0} \inf\Bigg\{\liminf_{n\to\infty}\bigg[}+ \int_{Q_\nu \cap \eps^{-1}(S_{u_n}-x_0)} \psi(x_0+\eps y, [u_n](x_0+\eps y),\nu_{u_n}(x_0+\eps y)) \,\de\mathcal H^{N-1}(y)\bigg] : \nonumber\\
&\!\!\!\!&\!\!\!\!
\phantom{\limsup_{\varepsilon \to 0}\frac{1}{\varepsilon^{N-1}} \inf\Bigg\{} \text{$u_n \in SBV(Q_\nu(x_0,\varepsilon);\mathbb R^d)$,  $u_n \to s_{\lambda,0,\nu}(\cdot-x_0)$ in $L^1(Q_\nu(x_0,\varepsilon);\mathbb R^d)$,} \nonumber \\
&\!\!\!\!&\!\!\!\!
\phantom{\limsup_{\varepsilon \to 0}\frac{1}{\varepsilon^{N-1}} \inf\Bigg\{} \text{$\nabla u_n \wsto 0$ in $\cM(Q_\nu(x_0,\varepsilon);\mathbb R^{d\times N})$} \Bigg\} \nonumber\\
&\!\!\!\!=&\!\!\!\!
\limsup_{\varepsilon \to 0} \inf\Bigg\{\liminf_{n\to\infty}\bigg[\int_{Q_\nu} \eps\, W(x_0+\eps y,\eps^{-1}\nabla v_n(y)) \, \de y + \int_{Q_\nu \cap S_{v_n}} \!\!\!\! \psi(x_0+\eps y, [v_n](y),\nu_{v_n}(y)) \,\de\mathcal H^{N-1}(y)\bigg] : \nonumber\\
&\!\!\!\!&\!\!\!\!
\phantom{\limsup_{\varepsilon \to 0}\frac{1}{\varepsilon^{N-1}} \inf\Bigg\{} \text{$v_n \in SBV(Q_\nu;\mathbb R^d)$,  $v_n \stackrel[n]{\longrightarrow}{} s_{\lambda,0,\nu}$ in $L^1(Q_\nu;\mathbb R^d)$, $\nabla v_n \stackrel[n]{*}{\longrightharpoonup} 0$ in $\cM(Q_\nu;\mathbb R^{d\times N})$} \Bigg\} \nonumber\\
&\!\!\!\!=&\!\!\!\!
\limsup_{\varepsilon \to 0} \inf\Bigg\{\liminf_{n\to\infty}\bigg[\int_{Q_\nu} \eps\, W(x_0+\eps y,\eps^{-1}\nabla v_n(y)) \, \de y + \int_{Q_\nu \cap S_{v_n}} \psi(x_0, [v_n](y),\nu_{v_n}(y)) \,\de\mathcal H^{N-1}(y)\bigg] : \nonumber\\
&\!\!\!\!&\!\!\!\!
\phantom{\limsup_{\varepsilon \to 0}\frac{1}{\varepsilon^{N-1}} \inf\Bigg\{} \text{$v_n \in SBV(Q_\nu;\mathbb R^d)$,  $v_n \stackrel[n]{\longrightarrow}{} s_{\lambda,0,\nu}$ in $L^1(Q_\nu;\mathbb R^d)$, $\nabla v_n \stackrel[n]{*}{\longrightharpoonup} 0$ in $\cM(Q_\nu;\mathbb R^{d\times N})$} \Bigg\} \nonumber\\
&\!\!\!\!=&\!\!\!\!
\limsup_{\varepsilon \to 0} \inf\Bigg\{\liminf_{n\to\infty} \widetilde{E}_{x_0,\eps}(v_n;Q_\nu) : \text{$v_n \in SBV(Q_\nu;\mathbb R^d)$,  $v_n \stackrel[n]{\longrightarrow}{} s_{\lambda,0,\nu}$ in $L^1(Q_\nu;\mathbb R^d)$,} \nonumber\\ 
&\!\!\!\!&\!\!\!\!
\phantom{\limsup_{\varepsilon \to 0}\inf\Bigg\{\liminf_{n\to\infty} \widetilde{E}_{x_0,\eps}(v_n;Q_\nu) : } \text{$\nabla v_n \stackrel[n]{*}{\longrightharpoonup} 0$ in $\cM(Q_\nu;\mathbb R^{d\times N})$} \Bigg\}= \tilde{h}_{1,1}(x_0,\lambda,\nu), \label{oh_no}
\end{eqnarray}
where we have used a change of variables, we have defined $v_n(y)\coloneqq u_n(x_0+\eps y)$, and we have invoked Remark~\ref{rem3.2}(d).

Now define $\widehat{m}_{1,1}^{\surface}(\lambda,\nu;Q_\nu(x_0,\varepsilon))\coloneqq \inf\big\{E(u;Q_\nu(x_0,\varepsilon)): u(x_0 + \varepsilon \cdot) \in \mathcal C^\surface(\lambda,\nu)\big\}$.
It is easy to verify that for a.e.~$x_0\in \Omega$, every $\lambda\in\R{d}$, and $\nu\in\mathbb{S}^{N-1}$ the following inequality holds true
$$m_{1,1}(s_{\lambda,0,\nu}(\cdot-x_0),0;Q_\nu(x_0,\varepsilon))\leq \widehat{m}_{1,1}^{\surface}(\lambda,\nu;Q_\nu(x_0,\varepsilon));$$
indeed, for any function $u$ that is admissible for $\widehat{m}_{1,1}^{\surface}(\lambda,\nu;Q(x_0,\varepsilon))$, 
the pair $(u,\nabla u)$ is a competitor for $m_{1,1}(s_{\lambda,0,\nu}(\cdot-x_0),0;Q(x_0,\varepsilon))$.
Consequently, by \eqref{oh_no},
\begin{align*}
\widetilde{h}_{1,1}(x_0,\lambda,\nu) = \limsup_{\varepsilon \to 0} \frac{m_{1,1}(s_{\lambda,0,\nu}(\cdot-x_0), 0;Q_\nu(x_0,\varepsilon))}{\varepsilon^{N-1}}
\leq \limsup_{\varepsilon \to 0} \frac{\widehat{m}_{1,1}^{\surface}(\lambda,\nu;Q_\nu(x_0,\varepsilon))}{\varepsilon^{N-1}}=h_{1,1}(x_0,\lambda,\nu),
\end{align*}
where in the last equality we used again a change of variables argument; therefore~\eqref{clean2} is proved.
\end{proof}
We stress here that Theorem~\ref{thm_approx_CF} generalizes the results of \cite[Theorems~2.16 and~2.17]{CF1997}, \cite[Theorem~3.2]{book}, and \cite[Theorems~4.1 and~4.2]{BMZ2024} in the following respect: the densities are $x$-dependent and the regularity of the initial bulk energy density $(x,A)\mapsto W(x,A)$ is weakened to Carath\'{e}odory. 
Moreover, we note that the more general expressions in formulas~\eqref{906} and~\eqref{907} generalize all those previously obtained in the references just cited.

\section{The relaxation theorem for $3$-level structured deformations}\label{multi}
In this section, we present the relaxation theorem for the specific case of $3$-level (first-order) structured deformations (that is for $L=2$), Theorem~\ref{mainL} below.
The proof we present relies on an equivalence of the complete relaxation of the initial energy \eqref{103} for the case $L=2$ (see \eqref{Ip2} below) and an iterated one (see \eqref{Ip2tilde} below) where one first relaxes to $2$-level structured deformations as in \eqref{102}, and then relaxes once more, see Proposition~\ref{T003}.
Our current proof is different from the one we found in \cite{BMMOZ}, since here we abandon freezing one of the variables at each step (see \cite[Section~3.2]{BMMOZ} for details).

Our choice for presenting the explicit details of the proof for this particular case is motivated by the fact that the case $L=2$ already contains the essential features of the general case $L\geq 2$.

The strategy of our proof is the following: after proving the equivalence of the two energies $I_{2,p}$ and $\widetilde{I}_{2,p}$\,, we study the properties of the relaxed energy densities $H_{1,p}$ and $h_{1,p}$ defined in \eqref{906} and \eqref{907}, respectively: Theorem~\ref{thm_propdens} below shows that some properties of the initial energy densities $(W,\psi)\in\ED(p)$ are maintained, whereas some are lost, especially in the case $p>1$. This theorem is needed to make sure that the relaxed energy densities obtained after the first relaxation are ``good'' energy densities on which to perform the second relaxation.

Lemmas~\ref{ubIp} and~\ref{nsa} contain technical results that show that the localization $\cO(\Omega)\ni O\mapsto I_{2,p}(\cdot,\cdot,\cdot;O)$ is the restriction of a Radon measure to the class $\cO(\Omega)$ of the open subsets of $\Omega$, see Proposition~\ref{radon}. 
Only at this point will it be possible to apply the global method for relaxation to the functional $I_{2,p}$\,, see Theorem~\ref{mainL}.

We refer the reader to Appendix~\ref{app_L=3} for some details of the case of $4$-level structured deformations, with the intention that these details should allow the reader independently to carry out further necessary relaxation steps.

Let the initial energy $E$ be given as in \eqref{103} and let $(g,G_1,G_2)\in SD_{2,p}(\Omega)$ (for $p\geq1$); we seek an integral representation for the relaxed energy
\begin{equation}\label{Ip2}
I_{2,p}(g,G_1, G_2)\coloneqq \inf\Big\{  \liminf_{n_1}\liminf_{n_2} E(u_{n_1, n_2}): \big(u_{n_1, n_2}\big)\subset SBV(\Omega;\R{d}), u_{n_1, n_2}\weakstH(g,G_1,G_2) \Big\};
\end{equation}
we also define the \emph{iterated} relaxation 
\begin{equation}\label{Ip2tilde}
\widetilde{I}_{2,p} (g, G_1, G_2)  \coloneqq \inf \Big\{  \liminf_{n_1}  I_{1,p}(\gamma_{n_1},\Gamma_{n_1}): \big((\gamma_{n_1},\Gamma_{n_1})\big)\subset SD_{1,p}(\Omega), \gamma_{n_1}\weakstH (g,G_1), \Gamma_{n_1}\wsto G_2 \Big\},
\end{equation}
where $I_{1,p}$ is defined in \eqref{102}.

The next proposition shows that the two relaxation processes in~\eqref{Ip2} and~\eqref{Ip2tilde} give the same result.
\begin{proposition}\label{T003}
Under Assumptions \ref{gen_ass}, for every $(g,G_1,G_2)\in SD_{2,p}(\Omega)$,
\begin{equation}
I_{2,p} (g, G_1, G_2) = \widetilde{I}_{2,p} (g, G_1, G_2) .
\end{equation}
\end{proposition}
\begin{proof}
Let $\delta > 0$ and consider $\big((\gamma_{n_1},\Gamma_{n_1})\big) \subset SD_{1,p}(\Omega)$ such that  $ \gamma_{n_1}\weakstH (g,G_{1})$, $\Gamma_{n_1} \wsto G_2$ and
\begin{equation}\label{eq1-4.4}
\widetilde{I}_{2,p}(g,G_1,G_2) + \delta \geq \liminf_{n_1} I_{1,p}(\gamma_{n_1},\Gamma_{n_1}).
	\end{equation}
By Corollary \ref{S004}, there exists an $SBV(\Omega;\R{d})$ sequence  $(n_1,n_2)\mapsto u_{n_1,n_2}=v_{n_2}^{(n_1)}$ (given by \eqref{defappr}, satisfying \eqref{apprGamma}) such that $u_{n_1, n_2}=v^{(n_1)}_{n_2}\weakstH (g,G_1,G_2)$ and is a recovery sequence for $I_{1,p}(\gamma_{n_1},\Gamma_{n_1})$ as $n_2 \to \infty$, i.e.,
\begin{equation}\label{Idputaformula}
I_{1,p}(\gamma_{n_1},\Gamma_{n_1})=\lim_{n_2} E\big(v_{n_2}^{(n_1)}\big)=\inf \Big\{\liminf_{n_2}E(v_{n_2}) : (v_{n_2}) \subset SBV(\Omega;\R{d}), v_{n_2}\weakstH (\gamma_{n_1},\Gamma_{n_1}) \Big\}.
\end{equation}
Observe that the sequence $\big(v_{n_2}^{(n_1)}\big)$ is admissible for $I_{2,p}$, since 
\begin{align*}
\lim_{n_1}\Big(\lim_{n_2}v_{n_2}^{(n_1)}\Big) &= \lim_{n_1}\gamma_{n_1}= g, \,
\mbox{ strongly in } L^1(\Omega;\R{d\times N}),\\
\lim_{n_1}\nabla \Big(\lim_{n_2} v_{n_2}^{(n_1)}\Big) &= \lim_{n_1}\nabla \gamma_{n_1}= G_1, \, \mbox{ weakly* in } \cM(\Omega;\R{d\times N}),\\
\lim_{n_1} \Big(\lim_{n_2} \nabla v_{n_2}^{(n_1)} \Big) &= \lim_{n_1} \Gamma_{n_1}= G_2, \,
\mbox{ weakly* in } \cM(\Omega;\R{d\times N}).
\end{align*}
Thus, by \eqref{Ip2}, \eqref{eq1-4.4}, and \eqref{Idputaformula}, we have
\begin{align*}
I_{2,p}(g,G_1,G_2)\leq &\, \liminf_{n_1}\liminf_{n_2}E\big(v^{(n_1)}_{n_2}\big) =  \liminf_{n_1} I_{1,p}(\gamma_{n_1},\Gamma_{n_1}) \leq \widetilde{I}_{2,p}(g,G_1,G_2) + \delta.
\end{align*}
Letting $\delta \to 0$, we conclude that $I_{2,p}(g,G_1,G_2) \leq \widetilde{I}_{2,p}(g,G_1,G_2)$. 

To prove the opposite inequality, we notice that for any sequence $\big(u_{n_1,n_2}\big) \subset SBV(\Omega;\R{d})$ such that $u_{n_1, n_2} \weakstH (g,G_1,G_2)$,  there exists $\big((\gamma_{n_1}, \Gamma_{n_1})\big) \subset SD_{1,p}(\Omega)$ such that $\displaystyle \lim_{n_2}u_{n_1,n_2}=\gamma_{n_1}$ in $L^1(\Omega;\R{d\times N})$, $\displaystyle\lim_{n_2}\nabla u_{n_1,n_2}= \Gamma_{n_1}$ weakly* in $\cM(\Omega;\mathbb R^{d \times N})$, $\gamma_{n_1}\weakstH(g,G_1)$, and  $\Gamma_{n_1} \wsto G_2$.
Hence, by \eqref{Ip2tilde} and \eqref{102}, 

$$\widetilde I_{2,p}(g,G_1,G_2) \leq \liminf_{n_1} I_{1,p}(\gamma_{n_1},\Gamma_{n_1}) \leq \liminf_{n_1} \Big(\liminf_{n_2}E(u_{n_1,n_2})\Big).$$
Taking the infimum over all such sequences $(u_{n_1,n_2})$ we obtain $\widetilde{I}_{2,p}(g,G_1,g_2)\leq I_{2,p}(g,G_1,G_2)$.
\end{proof}

In the sequel it will be useful to consider the localized version of the functional
$I_{2,p}(g,G_1,G_2)$, denoted by $I_{2,p}(g,G_1,G_2;O)$ and defined, for any open subset $O$ of $\Omega$, in full analogy with
$I_{2,p}(g,G_1,G_2)$ except that the admissible sequences are defined in $O$, the convergences in the sense of Definition \ref{S000} hold in $O$ and in the energy $E(\cdot)$ the integrations are performed over $O$.

To obtain our integral representation result for $I_{2,p}(g,G_1,G_2)$, we need some auxiliary results. We begin by proving some properties of the relaxed densities $H_{1,p}$ and $h_{1,p}$ that will be useful in the sequel.

\begin{theorem}[Properties of the relaxed energy densities]\label{thm_propdens}
Let $p\geq1$ and let $(W,\psi)\in\ED(p)$. 
Let $H_{1,p}\colon\Omega\times\R{d\times N}\times\R{d\times N}\to[0,+\infty)$ and $h_{1,p}\colon\Omega\times\R{d}\times\S{N-1}\to[0,+\infty)$ be the functions defined in~\eqref{906} and~\eqref{907}, respectively.
Then the function $H_{1,p}$ is $p$-Lipschitz continuous in the third component, namely for every $A\in \R{d \times N}$ there exists a constant $C>0$ such that for almost every $x\in\Omega$ and for every $B_1,B_2\in\R{d\times N}$,
\begin{equation}\label{H_B}
|H_{1,p}(x,A,B_1)-H_{1,p}(x,A,B_2)|\leq C|B_1-B_2|(1+|B_1|^{p-1}+|B_2|^{p-1});
\end{equation}
finally, there exist constants $\bar{c}_H,\overline{C}_{H}>0$ such that for almost every $x\in\Omega$ and for every $A,B\in\R{d\times N}$
\begin{equation}\label{JoE2.27}
\bar{c}_H (|A|+|B|^p)-\frac1{\bar{c}_H} \leq H_{1,p}(x,A,B) \leq \overline{C}_H (1+|A|+|B|^p).
\end{equation}
Let, for $B\in\R{d\times N}$, the function $H_{1,p}^B \colon\Omega\times\R{d\times N}\to[0,+\infty)$ be defined by $(x,A)\mapsto H_{1,p}^B(x,A)\coloneqq H_{1,p}(x,A,B)$.
Then
\begin{itemize}
\item[(i)] if $p>1$, then $(H_{1,p}^B,h_{1,p}) \in \ED(1)$;

\item[(ii)] if $p=1$, the function $H_{1,1}^B$ satisfies properties \eqref{W_cara}--\eqref{Wcoerc} of Assumptions~\ref{gen_ass} with $q=1$ and the function $h_{1,1}$ satisfies properties \eqref{(psi0)} and \eqref{(psi1)} of Assumptions~\ref{gen_ass}. 
Moreover, for a.e.~$x\in\Omega$ and for every $\nu\in\mathbb{S}^{N-1}$
\begin{equation}\label{h11Lip}
\big|{h}_{1,1}(x,\lambda_1,\nu)-{h}_{1,1}(x,\lambda_2,\nu)\big|\leq C_\psi|\lambda_1-\lambda_2|.
\end{equation}
\end{itemize}
\end{theorem}
\begin{proof}
Estimates~\eqref{H_B} and~\eqref{JoE2.27} can be proved in the same way as estimates (2.26) and (2.27) in \cite{BMMOZ}, which were proved for the case $p>1$ and can be easily adapted to cover the case $p=1$.

\noindent (i): The thesis can be proved exactly in the same way as \cite[Theorem 2.10]{BMMOZ}, the only difference being in the regularity assumptions on~$W$, which now is only measurable with respect to the $x$-variable (see Assumption~\ref{gen_ass}-\eqref{W_cara}). 
In the present setting, measurability of $H_{1,p}^B$ is granted by the fact that $H_{1,p}$ is a Radon--Nikodym derivative.

\noindent (ii): The results concerning $H_{1,1}^B$ can be easily adapted from the proof of \cite[Theorem 2.10]{BMMOZ}, whereas those concerning $h_{1,1}$ require more care.

The symmetry property \eqref{(psi0)} is immediate. 
To prove that $h_{1,1}$ satisfies the growth condition from above, it suffices to consider an admissible $u\in \cC^{\surface}(\lambda,\nu)$ such that $\nabla u=0$ a.e.~in~$\Omega$ and to apply properties \eqref{(W1)_p} with $q=1$, \eqref{W1bdd}, and \eqref{(psi1)} (estimate from above); to prove the estimate from below, it suffices to use the fact that $W\geq0$ and to apply property \eqref{(psi1)} (estimate from below).

The proof of~\eqref{h11Lip} can be achieved using the definitions of $\limsup$ and $\inf$ and applying the Lipschitz continuity of~$\psi$ (see \eqref{psi_Lip}).
\end{proof}

\color{black}

\color{black}
\begin{lemma}\label{ubIp}
Let $p\geq1$ and assume that \eqref{(W1)_p} (with $q=p$), \eqref{W1bdd}, and \eqref{(psi1)} in Assumptions~\ref{gen_ass} hold true. 
Then there exists a constant $C > 0$ such that, for any $(g,G_1,G_2)\in SD_{2,p}(\Omega)$
and for every $O \in \cO(\Omega)$, 
\begin{align*}
&\frac{1}{C} \Big[|Dg|(O) + 
\|G_1\|_{L^1(O;\R{d \times N})} + \|G_2\|^p_{L^p(O;\R{d \times N})}\Big]
\leq 
I_{2,p}(g,G_1,G_2;O) \\
&\hspace{5cm}\leq C \Big[\mathcal L^N(O) + |Dg|(O) + 
\|G_1\|_{L^1(O;\R{d \times N})} + \|G_2\|^p_{L^p(O;\R{d \times N})}\Big].
\end{align*}
\end{lemma}
\begin{proof}
The sequence $(n_1,n_2)\mapsto u_{n_1,n_2}$, constructed in the proof of Theorem \ref{appTHMh} by means of Theorem \ref{Al} and Lemma \ref{ctap}, is admissible for $I_{2,p}(g,G_1,G_2;O)$. Therefore, using \eqref{(W1)_p} with $q=p$, \eqref{W1bdd}, \eqref{(psi1)}, \eqref{817}, and \eqref{818}, we obtain
$$I_{2,p}(g,G_1,G_2;O) \leq \liminf_{n_1,n_2}E(u_{n_1,n_2};O) \leq C \Big[\mathcal L^N(O) + |Dg|(O) + 
\|G_1\|_{L^1(O;\R{d \times N})} + \|G_2\|^p_{L^p(O;\R{d \times N})}\Big].$$
The proof of the lower bound follows the arguments presented in the proof of \cite[Theorem 2.10]{BMMOZ} to obtain the lower bounds for the relaxed energy densities. 
\end{proof}

\begin{lemma}[Nested sub-additivity]\label{nsa}
Let $p\geq1$ and assume that Assumptions~\ref{gen_ass} with $q=p$ hold and let $O_1, O_2, O_3\in\mathcal{O}(\Omega)$ be such that $O_1 \subset \subset O_2 \subseteq O_3$. 
Then, for every $(g,G_1,G_2)\in SD_{2,p}(\Omega)$,
$$I_{2,p}(g,G_1,G_2;O_3) \leq I_{2,p}(g,G_1,G_2;O_2) + I_{2,p}(g,G_1,G_2;O_3 \setminus \overline O_1).$$
\end{lemma}
\begin{proof}
The proof relies on the fact, proved in Proposition \ref{T003}, that  $I_{2,p} = \widetilde{I}_{2,p}$\,, by the definition of which
there exist sequences $u_n \in SBV(O_2;{\mathbb{R}}^d)$, $U_n \in L^p(O_2;{\mathbb{R}}^{d\times N})$, $v_n \in SBV(O_3\setminus \overline O_1;{\mathbb{R}}^d)$, and  $V_n \in L^p(O_3\setminus \overline O_1;{\mathbb{R}}^{d\times N})$ such that $u_n \rightarrow g$ in $L^1(O_2;{\mathbb{R}}^d)$, 
$\nabla u_n \wsto G_1$ and $U_n \wsto G_2$ in $\cM(O_2;{\mathbb{R}}^{d\times N})$,
$v_n \rightarrow g$ in $L^1(O_3\setminus \overline O_1;{\mathbb{R}}^d)$, 
$\nabla v_n \wsto G_1$ and $V_n \wsto G_2$ in $\cM(O_3\setminus \overline O_1;{\mathbb{R}}^{d\times N})$, and, in addition, by Theorem \ref{thm_approx_CF}, applied in $O_2$ and in $O_3\setminus \overline O_1$,
\begin{equation*}
I_{2,p}(g,G_1,G_2;O_2) = \lim_{n}\bigg[\int_{O_2}H_{1,p}(x, \nabla u_n(x), U_n(x)) \, \de x + \int_{S_{u_n}\cap O_2} h_{1,p}(x,[u_n](x),\nu_{u_n}(x)) \, \de\mathcal{H}^{N-1}(x)\bigg] 
\end{equation*}
and 
\begin{align*}
I_{2,p}(g,G_1,G_2;O_3\setminus \overline{O}_1) &= \lim_{n} \bigg[\int_{O_3\setminus \overline{O}_1}H_{1,p}(x, \nabla v_n(x), V_n(x)) \, \de x \\
&\hspace{4cm} + \int_{S_{v_n}\cap(O_3\setminus \overline{O}_1)} h_{1,p}(x,[v_n](x),\nu_{v_n}(x))\, \de\mathcal{H}^{N-1}(x)\bigg].
\end{align*}
Notice that 
\begin{eqnarray}  
u_n - v_n \rightarrow 0 & \text{in} & L^1(O_2 \cap (O_3\setminus \overline{O}_1);{\mathbb{R}}^d), \label{L1conv} \\
\nabla u_n - \nabla v_n \wsto 0 & \text{in} & \cM(O_2 \cap(O_3\setminus \overline{O}_1);{\mathbb{R}}^{d\times N}), \nonumber \\
U_n - V_n \wsto 0 & \text{in} & \cM(O_2 \cap(O_3\setminus \overline{O}_1);{\mathbb{R}}^{d\times N}). \nonumber
\end{eqnarray}

For $\delta > 0$, define $O_{\delta} \coloneqq \{ x \in O_2: d(x) < \delta\}$, where, for every $x\in O_3$, we define the function $d(x)\coloneqq \mbox{dist}(x, O_1)$.
Since the distance function to a fixed set is Lipschitz continuous, we can apply the change of variables formula \cite[Section 3.4.3, Theorem 2]{EG}, to obtain 
\begin{equation*}
\int_{O_{\delta}\setminus \overline O_1} |u_n(x) - v_n(x)| \, |\det\nabla d(x)| \, \de x = \int_{0}^{\delta}\left [ \int_{d^{-1}(y)} |u_n(x) - v_n(x)| \, \de \mathcal{H}^{N-1}(x)\right]\, \de y
\end{equation*}
and, as $|\det\nabla d|$ is bounded and \eqref{L1conv} holds, it
	follows, by Fatou's Lemma, that for almost every $\rho \in [0, \delta]$ we have 
	\begin{equation}  \label{aer}
		\liminf_{n} 
		\int_{d^{-1}(\rho)} |u_n(x) - v_n(x) |\, \de\mathcal{H}^{N-1}(x) 
		= \liminf_{n} \int_{\partial O_{\rho}}
		|u_n(x) - v_n(x) |\, \de\mathcal{H}^{N-1}(x) = 0.
	\end{equation}
	Consider the (not relabelled) subsequences of $u_n$ and $v_n$ for which the liminf in \eqref{aer} is attained. Fix $\rho_0 \in [0, \delta]$ such that 
	$\|G_1 \chi_{O_2}\|_{L^1(\partial O_{\rho_0})} =\|G_2 \chi_{O_2}\|_{L^p(\partial O_{\rho_0})}= 0$, 
	$\|G_1 \chi_{O_3 \setminus \overline O_1}\|_{L^1(\partial O_{\rho_0})} =\|G_2 \chi_{O_3 \setminus \overline O_1}\|_{L^p(\partial O_{\rho_0})}= 0$ and
	such that \eqref{aer} holds. We observe that $O_{\rho_0}$ is a set with
	locally Lipschitz boundary since it is a level set of a Lipschitz function
	(see, e.g., \cite{EG}). Hence we can consider $u_n, v_n$ 
	on $\partial O_{ \rho_0}$ in the sense of traces and define 
	\begin{equation*}
		w_n \coloneqq 
		\begin{cases}
			u_n & \text{ in}\; \overline{O}_{\rho_0} \\ 
			v_n & \text{ in}\; O_3\setminus \overline{O}_{\rho_0}
		\end{cases}
\qquad\text{and}\qquad
		Z_n \coloneqq 
		\begin{cases}
			U_n & \text{ in}\; \overline{O}_{\rho_0} \\ 
			V_n & \text{ in}\; O_3\setminus \overline{O}_{\rho_0}.
		\end{cases}
	\end{equation*}
	\noindent Then, by the choice of $\rho_0$, $(w_n, Z_n)$ is admissible for 
	$I_{2,p}(g,G_1,G_2;O_3)$ and by the linear growth property of $h_{1,p}$ (see  Theorem \ref{thm_propdens}), \eqref{L1conv} and \eqref{aer}, we obtain 
	\begin{align*}
		I_{2,p}(g,G_1,G_2;O_3) &\leq  \liminf_{n} 
		\left[\int_{O_3}H_{1,p}(x, \nabla w_n(x), Z_n(x)) \, \de x +
		\int_{S_{w_{n}}\cap O_3} h_{1,p}(x,[w_n](x),\nu_{w_n}(x))\, \de\cH^{N-1}(x)\right] \\
		&\leq  \liminf_{n} 
		\left[\int_{O_2}H_{1,p}(x, \nabla u_n(x), U_n(x)) \, \de x +
		\int_{S_{u_{n}}\cap O_2} h_{1,p}(x,[u_n](x),\nu_{u_n}(x))\, \de\cH^{N-1}(x) \right. \\
		&  \hspace{.7cm} + \int_{O_3\setminus \overline O_1}H_{1,p}(x, \nabla v_n(x),V_n(x)) \, \de x  + \int_{S_{v_{n}}\cap (O_3 \setminus \overline O_1)}
		h_{1,p}(x,[v_n](x),\nu_{v_n}(x))\, \de\cH^{N-1}(x) \\
		&  \left. \hspace{.7cm} + \int_{S_{w_n} \cap \partial O_{\rho_0}} 
		C |u_n(x) - v_n(x)| \, \de\mathcal{H}^{N-1}(x) \right] \\
		& =  I_{2,p}(g,G_1,G_2;O_2) + I_{2,p}(g,G_1,G_2;O_3 \setminus \overline O_1),
	\end{align*}
which concludes the proof.
\end{proof}

\begin{proposition}\label{radon}
Under Assumptions \ref{gen_ass} with $q=q$, for every $(g, G_1,G_2)\in  SD_{2,p}(\Omega)$, the restriction of $I_{2,p}(g,G_1,G_2;\cdot)$ to $\cO(\Omega)$ is a Radon measure, absolutely continuous with respect to 
$\cL^N + \cH^{N-1}\res{S(g)}$, i.e. for every $O \in \mathcal O(\Omega)$, there exists $C>0$
$$
I_{2,p}(g, G_1, G_2; O)\leq C \int_O (1+ |G_1|+ |G_2|^p)\, \de x +|D g|(O),
$$
\end{proposition}
\begin{proof}
Using the fact that $I_{2,p}= \widetilde{I}_{2,p}$, the  conclusion is achieved as in \cite[Lemma~2.22]{CF1997}, relying strongly on the nested sub-additivity property given in Lemma \ref{nsa} and on the upper bound obtained in Lemma \ref{ubIp}.
\end{proof}

We now state our main result for $3$-level structured deformations.
\begin{theorem}\label{mainL}
Let $p \geq1$, $(W,\psi)\in\cE\cD(p)$, and  $(g,G_1,G_2)\in SD_{2,p}(\Omega)$. 
Then $I_{2,p}(g, G_1, G_2)$, defined by \eqref{Ip2}, admits the integral representation 
\begin{equation*}\label{596}
\begin{split} 
I_{2,p}(g, G_1, G_2)& = \int_\Omega f_{2,p}(x,\nabla g(x), G_1(x), G_2(x)) \,\de x + \int_{\Omega \cap S(g)} \Phi_{2,p}(x, [g](x), \nu_g(x))\, \de \cH^{N-1}(x),\\
\end{split}
\end{equation*}
where, for a.e.~$x_0 \in \Omega$ and for all $(\xi, B_1,B_2) \in \R{d\times N}\times \R{d\times N}\times \R{d\times N},$ $\lambda \in \mathbb R^d, \nu \in \mathbb S^{N-1}$,
the densities $f_{2,p}(x_0, \xi, B_1,B_2)$  and $\Phi_{2,p}(x_0, \lambda, \nu)$ are given by 
\begin{align}\label{f2p}
	f_{2,p}(x_0, \xi, B_1, B_2) \coloneqq \limsup_{\varepsilon \to 0^+}
	\frac{m_{L,p}( a_{\xi,x_0}, B_1, B_2; Q(x_0,\varepsilon))}{\varepsilon^N},
\end{align}
\begin{align}\label{Phi2p}
	\Phi_{2,p}(x_0, \lambda, \nu) \coloneqq \limsup_{\varepsilon \to 0^+}\frac{m_{2,p}(s_{\lambda, 0,\nu}(\cdot-x_0), 0, 0; Q_{\nu}(x_0,\varepsilon))}{\varepsilon ^{N-1}},
\end{align}
where $m_{2,p}$ is given by \eqref{mLdef} (for $L=2$) and where the functional $\mathcal{F}$ is taken to be $I_{2,p}$\,.
\end{theorem}
\begin{proof}
Notice that $I_{2,p}\colon SD_{2,p}(\Omega) \times\mathcal{O}(\Omega)\to [0, +\infty]$ satisfies assumptions $(H1)$-$(H4)$ above. 
Indeed, $(H1)$ was proved in Proposition~\ref{radon} and $(H2)$ follows from the definition of $I_{2,p}$\,.
Property $(H4)$ was proved in Lemma \ref{ubIp}, whereas $(H3)$ is, by standard arguments, an immediate consequence of $(H2)$. In view of Theorem \ref{globalmethod-new} and Remark \ref{remtrasl}, $I_{2,p}$ admits the stated integral representation.
\end{proof}

\begin{remark}\label{remdens}
\begin{itemize}
\item[(a)] Following the strategy of the proof of Theorem \ref{thm_approx_CF}, it is reasonable to expect that, for every $p\geq 1$,
	\begin{equation*}\label{bulk2p}
		\begin{split}
			f_{2,p}(x,A, B, C) & = \limsup_{\varepsilon \to 0} \inf \bigg\{ 
			\int_Q H_{1,p}(x + \varepsilon y,\nabla u(y), \Gamma(y))\, \de y + 
			\int_{Q \cap S(u)} \hspace{-0,3cm}h_{1,p}(x,  [u](y), \nu_u(y)) \, \de \cH^{N-1}(y) :\\
			& \hspace{1,5cm} u \in SBV(Q; \R{d}), \; \Gamma \in L^{1}(Q; \R{d\times N}), \;
			\int_Q \nabla u \, \de y = B, \; u = \ell_A \; \text{on}\; \partial Q, \;  
			\int_Q \Gamma \, \de y = C \bigg\},
		\end{split}
	\end{equation*}
whose proof could follow the steps of the proof of Theorem~\ref{thm_approx_CF}; we decided to omit the explicit computations, which are extensive and do not add any novelty to the mathematics of this paper.

\item[(b)] Proposition \ref{radon} guarantees that, for every $(g, G_1, G_2)\in SD_{2,p}(\Omega)$, the computation of the Radon--Nikodym derivative $\displaystyle \frac{\de I_{2,p}(g, G_1, G_2;  \cdot)}{\de |D^s g|}(x_0)$, at $x_0 \in S_g$, does not depend on $G_1, G_2$\,.
To verify this assertion, we follow the arguments presented in \cite[Section 4.2]{AMMZ}, which rely exclusively on the Lipschitz behaviour of $H_{1,p}(x, \cdot, B)$ and $h_{1,p}(x, \cdot, \nu)$ and the $p$-Lipschitz behaviour of $H_{1,p}(x,A, \cdot)$. 
	
Let $U\in \mathcal O(\Omega)$, consider $(\gamma_{n}, \Gamma_{n})\subset SD_{1,p}$  as a recovery sequence for $I_{2,p}(g, G_1, G_2; U)$ in the sense of Remark \ref{remLp} and, by Theorems~\ref{Al} and~\ref{ctap}, let us consider $v\in SBV(U;\mathbb R^d)$ such that $\nabla v=-G_1$ and piecewise constant functions $v_n\in SBV(U;\mathbb R^d)$ such that $v_n\to v$ in $L^1(U;\mathbb R^d)$. 
	Finally, let us define $\Gamma'_{n}= \Gamma_n- G_2$, $\gamma'_{n}\coloneqq \gamma_n+v-v_n$, so that  $(\gamma'_n, \Gamma'_n) \weakstH (g, 0,0)$, with  $\Gamma_n' \pweak 0$, and therefore

\[\begin{split}
&\, I_{2,p}(g,0,0 ;U)-I_{2,p}(g, G_1,G_2;U)\\
\leq&\,  \liminf_{n \to\infty} \bigg\{ \int_U \bigg(H_{1,p}\Big(x, \nabla \gamma'_n(x), \Gamma'_n(x)\Big)- H_{1,p}\Big(x, \nabla \gamma_n(x), \Gamma_n(x)\Big)\bigg)\de x \\
&\,\phantom{\liminf_{n \to\infty} \bigg\{} +\int_{U \cap S_{\gamma'_n}}h_{1,p}\Big(x, [\gamma'_n](x), \nu_{\gamma'_n}(x)\Big)\de\cH^{N-1}(x)-\int_{U \cap S_{\gamma_n}}h_{1,p}\Big(x, [\gamma_n](x), \nu_{\gamma_n}(x)\Big)\de \cH^{N-1}(x)\bigg\} \\
=&\, \liminf_{n \to\infty} \bigg\{ \int_U \bigg(H_{1,p}\Big(x, \nabla \gamma'_n(x), \Gamma'_n(x)\Big)- H_{1,p}\Big(x, \nabla \gamma_n(x), \Gamma'_n(x)\Big)\bigg)\de x \\
&\,\phantom{\liminf_{n \to\infty} \bigg\{} +\int_U \bigg(H_{1,p}\Big(x, \nabla \gamma_n(x), \Gamma'_n(x)\Big)- H_{1,p}\Big(x, \nabla \gamma_n(x), \Gamma_n(x)\Big)\bigg)\de x \\
&\,\phantom{\liminf_{n \to\infty} \bigg\{}+\int_{U \cap S_{\gamma'_n}}h_{1,p}\Big(x, [\gamma'_n](x), \nu_{\gamma'_n}(x)\Big)\de\cH^{N-1}(x)-\int_{U \cap S_{\gamma_n}}h_{1,p}\Big(x, [\gamma_n](x), \nu_{\gamma_n}(x)\Big)\de \cH^{N-1}(x)\bigg\}.
\end{split}\]
Hence, exploiting the triangle inequality, \eqref{H_B}, and the fact that $H^B_{1,p}$ satisfies \eqref{W_cara}-\eqref{Wcoerc} of Assumptions \ref{gen_ass} for $q=1$, and $h_{1,p}$ satisfies \eqref{psi_Lip} (see Theorem \ref{thm_propdens}), using H\"{o}lder inequality, we estimate the above energy as follows
\[\begin{split}
&\, I_{2,p}(g,0,0 ;U)-I_{2,p}(g, G_1,G_2;U) \\
\leq&\,\liminf_{n\to \infty} C\bigg\{\int_U C |\nabla \gamma_n(x)-\nabla \gamma_n'(x)|\, \de x \\
&\,\phantom{\liminf_{n\to \infty} C\bigg\{}  +\int_U C |\Gamma'_n(x)-\Gamma_n(x)|\left(1+ |\nabla \gamma_n(x)|^{\frac{p-1}{p}}(x)+ |\Gamma_n(x)|^{p-1}+ |\Gamma'_n(x)|\right))^{p-1}\,\de x \\
&\,\phantom{\liminf_{n\to \infty} C\bigg\{}  +\int_{U \cap S_v}|[v](x)| \,\de\cH^{N-1}(x) +\int_{U \cap S_{v_n}}|[v_n](x)|\,\de \cH^{N-1}(x)\bigg\}\\
\leq&\,\liminf_{n\to \infty} C\bigg\{ \int_U  |\nabla \gamma_n(x)-\nabla \gamma_n'(x)|\, \de x +\int_U C |\Gamma'_n(x)-\Gamma_n(x)| \, \de x\\
&\,\phantom{\liminf_{n\to \infty} C\bigg\{}  + \left(\int_{U} |\nabla \gamma_n(x)| \, \de x\right)^{\frac{p-1}{p}} \left(\int_\Omega |\Gamma_n(x)- \Gamma'_n(x)|^p\, \de x\right)^{\frac{1}{p}} +\int_U |\Gamma_n(x)|^{p} \, \de x+ \int_U|\Gamma'_n(x)|^{p}\,\de x\\ 
&\,\phantom{\liminf_{n\to \infty} C\bigg\{}  + \left(\int_U |\Gamma_n(x)|^{p}\, \de x\right)^{\frac{p-1}{p}}\left(\int_U |\Gamma'_n(x)|^p \, \de x\right)^{\frac{1}{p}} + \left(\int_U |\Gamma'_n(x)|^{p}\, \de x\right)^{\frac{p-1}{p}}\left(\int_U |\Gamma_n(x)|^p \, \de x\right)^{\frac{1}{p}}\\
&\,\phantom{\liminf_{n\to \infty} C\bigg\{}  + \int_{U \cap S_v}|[v](x)| \,\de\cH^{N-1}(x) +\int_{U \cap S_{v_n}}|[v_n](x)|\,\de \cH^{N-1}(x)\bigg\}\\
\leq&\,\liminf_{n\to \infty} C\bigg\{\int_U (1+ |G_1(x)|+ |G_2|^p(x))\,\de x + \int_{U \cap S_v}|[v](x)| \,\de\cH^{N-1}(x) +\int_{U \cap S_{v_n}}|[v_n](x)|\,\de \cH^{N-1}(x)\bigg\},
\end{split}\]
where the last inequality is a consequence of the weak* convergence of $\nabla \gamma_n$ and $\nabla \gamma_n'$ towards $G_1 \in L^1(\Omega;\mathbb R^{d \times N})$ and $0$, respectively, and the $L^p$-weak convergence of $\Gamma_n$ and $\Gamma'_n$ towards $G_2 \in L^p(\Omega;\mathbb R^d)$ and $0$, respectively, and 
where $C>0$ is a suitable constant, varying from line to line.

By virtue of the estimate in \eqref{817} and by \eqref{818}, the two surface integrals in the last line above are bounded by $\int_U |G_1(x)|\, \de x$ so that, 
	by exchanging the roles of $I_{2,p}(g,G_1, G_2; U)$ and $I_{2,p}(g, 0, 0; U)$,
 we arrive at the conclusion that
$$|I_{2,p}(g,0,0;U)-I_{2,p}(g, G_1, G_2;U)|\leq C\int_U (1+ |G_1|(x)|+ |G_2|^p(x))\,\de x,$$
for every $U \in \mathcal O(\Omega)$.
In turn, this guarantees that, for $\cH^{N-1}$-a.e.~$x_0 \in S_g$,
\begin{equation}\label{derFsur}
	\frac{\de I_{2,p}(g,0,0;\cdot)}{\de |D^s g|}(x_0)= \frac{\de I_{2,p}(g, G_1,G_2, \cdot)}{\de |D^s g|}(x_0).
\end{equation}	
In view of this, without loss of generality, we may consider $G_1=G_2=0$ when we compute the surface energy density. 
	
\item[(c)]  Observe that for every $p \geq 1$, for $\cH^{N-1}$ a.e.~$x_0$\,, and for all $\lambda \in \R{d}, \; \nu \in \cS^{N-1}$, 
	$$\Phi_{2,p}(x_0,\lambda, \nu) = h_{1,p}(x_0,\lambda, \nu),$$
	given by \eqref{907}.
	Indeed, by Proposition \ref{T003}, the definition of $\widetilde{I}_{2,p}$\,, and by using the lower semicontinuity with respect to  $L^1 \times \mathcal M_{{\rm weak}*} $ convergence of $I_{1,p}$\,, (that can be proved as in \cite[Proposition 5.1]{CF1997}), we have, for every $U \in \mathcal O(\Omega)$,
	\begin{equation*}
		\begin{split}
		I_{2,p}(g, G_1, G_2; U)& = 	\widetilde I_{2,p}(g, G_1, G_2; U) \geq \inf \{\liminf_{n_1} I_{1,p}(\gamma_{n_1}, \Gamma_{n_1})(U): \; \gamma_{n_1} \to g, \; \Gamma_{n_1} \wsto G_2\}\\
			& \geq I_{1,p}(g, G_2)(U)= \int_U H_{1,p}(x, \nabla g, G_2)\; \de x + \int_{U\cap S(g)} h_{1,p} (x, [g], \nu(g))\; \de \cH^{N-1}\\
		\end{split}
	\end{equation*}
	where in the last line we have exploited Theorem \ref{thm_approx_CF}.
	
Then, since, in view of Proposition \ref{radon}, $I_{2,p}(g, G_1, G_2; \cdot)$ is a measure which is absolutely continuous with respect to $\cL^N + \cH^{N-1}\res{S(g)}$,  it suffices to take the Radon--Nikodym derivative with respect to $\cH^{N-1}\res{S(g)}$, when $g=s_{\lambda,0, \nu}$\,, on both sides of the previous inequality, to obtain $	\Phi_{2,p}(x_0,\lambda, \nu) \geq h_{1,p}(x_0,\lambda, \nu)$. 

Regarding the reverse inequality, using again Theorem \ref{mainL} and (b)
	\begin{align}\label{RNx0}
	\Phi_{2,p}(x_0,\lambda, \nu )= 
	\frac{\de I_{2,p}(s_{\lambda,0,\nu}(\cdot-x_0), G_1, G_2; \cdot)}{\de {\mathcal H}^{N-1}\res S(s_{\lambda,0,\nu})}(x_0)= \frac{\de I_{2,p}(s_{\lambda,0,\nu}(\cdot-x_0), 0, 0; \cdot)}{\de {\mathcal H}^{N-1}\res S(s_{\lambda,0,\nu})}(x_0).
	\end{align}
	Taking into account the definition of $\widetilde{I}_{2,p}$ in \eqref{Ip2tilde} by taking $\gamma_{n_1}= s_{\lambda, 0, \nu}(\cdot-x_0)$, and  $\Gamma_{n_1} = 0$, and invoking Theorem \ref{thm_approx_CF},  
	\begin{equation*}
		\begin{split}
			I_{2,p}(s_{\lambda,0, \nu}(\cdot-x_0), 0, 0; U)& = 	\widetilde I_{2,p}(s_{\lambda,0, \nu}(\cdot-x_0), 0, 0; U) \leq \liminf_{n_1} I_{1,p}(s_{\lambda,0, \nu}(\cdot-x_0), 0; U) \\
			& =	\int_U H_{1,p}(x, 0, 0)\; \de x + \int_{U\cap S_{s_{\lambda,0,\nu}(\cdot-x_0)}}h_{1,p} (x, [\lambda], \nu))\; \de \cH^{N-1},\\
		\end{split}
	\end{equation*}
	which gives the desired inequality in view of \eqref{RNx0}. 
	\item[(d)] In \cite[Theorem 3.4]{BMMOZ} we proposed a recursive relaxation procedure to assign an energy to a three-level structured deformation in the case $p>1$, for $(g, G_1, G_2)\in SD_{2,p}(\Omega)$ with $\nabla g, G_1 \in L^p(\Omega;\mathbb R^{d\times N})$.
We point out that, in view of the growth conditions in \eqref{JoE2.27}, this is not the natural space in which to set the problem; therefore we could not apply \cite[Theorem 3.2]{BMZ2024} and for this reason we need to rely on Theorem \ref{globalmethod-new}. 
We also stress that, since $SD_{2,p}$ is a different space from $HSD^2_p$ in \cite{BMZ2024}, Theorem \ref{mainL} should yield a lower energy than the one in \cite[Theorem 3.4]{BMMOZ}, despite the fact that the surface energy density $\Phi_{2,p}=h_{1,p}$ in \eqref{907} coincides with the one obtained in \cite{BMMOZ}, as an easy computation reveals.  
Nevertheless, in some situations the bulk energy densities also coincide, as the example below shows.
\end{itemize}
\end{remark}

\begin{example}\label{example}
Let $p>1$ and consider an initial bulk energy density $W$ independent of the $x$-variable and convex, and the initial surface energy density $\psi (\lambda, \nu) = |\lambda \cdot \nu|$ (so, also independent of the $x$-variable).
Following the explicit example of \cite[Section~3.3]{BMMOZ} and \cite[Section~3.2.2.1]{book}, we have that
$$ H_{1,p}(A,B) = W(B) + |\tr (A -B)| \qquad\text{and}\qquad h_{1,p}(\lambda,\nu)=\psi(\lambda,\nu)=|\lambda\cdot\nu|.$$
Moreover, since $H_{1,p}$ is still convex in the first variable and both relaxed energy densities are independent of the $x$-variable, we can replicate the process, obtaining
$$H_{2,p} (A, B, C) = W(C) +  |\tr (B - C)| +  |\tr (B -A)|\qquad \text{and}\qquad h_{2,p}(\lambda,\nu)=h_{1,p}(\lambda,\nu)=|\lambda\cdot\nu|,$$
which is consistent with \cite[formula (3.15)]{BMMOZ}.
\end{example}

\appendix
\section{The case $L=3$}\label{app_L=3}
Under Assumptions \ref{gen_ass}, given the initial energy $E$ in \eqref{103} and given $(g,G_1,G_2,G_3)\in SD_{3,p}(\Omega)$, we seek an integral representation of the relaxed energy
\begin{equation}\label{S005}
\begin{split}
&\,I_{3,p}(g,G_1, G_2,G_3)\coloneqq \inf\Big\{  \liminf_{n_1,n_2,n_3} E(u_{n_1, n_2,n_3}): \big(u_{n_1, n_2,n_3}\big)\subset SBV(\Omega;\R{d}), u_{n_1, n_2,n_3}\weakstH(g,G_1,G_2,G_3)\Big\} 
\end{split}
\end{equation}
Moreover, define
\begin{equation}\label{T003eq}
\begin{split}
\widetilde{I}_{3,p} (g, G_1, G_2,G_3)  \coloneqq 
\inf \Big\{  \liminf_{n_1}\widetilde I_{2,p}(\gamma_{n_1},\Gamma_{n_1}, \Upsilon_{n_1}): \, &(\gamma_{n_1},\Gamma_{n_1}, \Upsilon_{n_1})\in SD_{2,p}(\Omega), \\
& \gamma_{n_1}\weakstH (g,G_1), \Gamma_{n_1}\weakst G_2, \Upsilon_{n_1}\weakst G_3 \Big\}, 
\end{split}
\end{equation}
where  $\widetilde I_{p,2}$ is the functional given by  \eqref{Ip2tilde}. 

We start by stating and proving the equivalent of Corollary \ref{S004} in the case $L=3$.
\begin{corollary}\label{K3}
For every $(g,G_1,G_2,G_3)\in SD_{3,p}(\Omega)$ and for every sequence $\big(\gamma_{n_1},\Gamma_{n_1},\Upsilon_{n_1}\big)\in SD_{2,p}(\Omega)$ such that $\gamma_{n_1}\weakstH(g,G_1)$ in the sense of Definition~\ref{S000}, 
$\Gamma_{n_1}\wsto G_2$ in $\cM(\Omega;\R{d\times N})$ and 
$\Upsilon_{n_1}\wsto G_3$ in $\cM(\Omega;\R{d\times N})$, there exists a sequence $(n_1,n_2,n_3)\mapsto u_{n_1,n_2,n_3}$ converging to $(g,G_1,G_2,G_3)$ 
in the sense of Definition~\ref{S000}.

Furthermore, if $p>1$ and
$\displaystyle \sup_{n_1} \|(\Gamma_{n_1},\Upsilon_{n_1})\|_{L^p(\Omega;\R{d \times N}\times \R{d \times N})} < + \infty$, then
$\displaystyle\sup_{n_1,n_2,n_3} \|\nabla u_{n_1,n_2,n_3}\|_{L^p(\Omega;\R{d \times N})} < + \infty.$
\end{corollary}
\begin{proof}
By Corollary \ref{S004}, for every $n_1$, there exists a sequence $(n_2,n_3)\mapsto v_{n_2,n_3}^{(n_1)}\in SBV(\Omega;\R{d})$ such that $v_{n_2,n_3}^{(n_1)}$ converges to $(\gamma_{n_1},\Gamma_{n_1},\Upsilon_{n_1})$ in the sense of Definition \ref{S000}.
This means that 
$$ \lim_{n_2} \lim_{n_3}v_{n_2,n_3}^{(n_1)} = \gamma_{n_1} 
\mbox{ strongly in } L^1(\Omega;\R{d}),$$
$$\lim_{n_3}v_{n_2,n_3}^{(n_1)} = g_{n_2}^{(n_1)} \; \mbox{ where }
g_{n_2}^{(n_1)} \in SBV(\Omega;\R{d}) \; \mbox{ is such that } 
\lim_{n_2}\nabla g_{n_2}^{(n_1)} = \Gamma_{n_1} \mbox{ weak * in }  \cM(\Omega;\R{d\times N}),$$
and finally
$$ \lim_{n_2} \lim_{n_3}\nabla v_{n_2,n_3}^{(n_1)} = \Upsilon_{n_1} 
\mbox{ weakly* in } \cM(\Omega;\R{d\times N}).$$
Then the sequence 
\begin{equation}
(n_1,n_2,n_3)\mapsto u_{n_1,n_2,n_3}\coloneqq v_{n_2,n_3}^{(n_1)}
\end{equation}
approximates $(g,G_1,G_2,G_3)$ in the sense of Definition~\ref{S000}.
Indeed, 
$$\lim_{n_1}\lim_{n_2}\lim_{n_3} u_{n_1,n_2,n_3}=\lim_{n_1}\lim_{n_2}\lim_{n_3} v_{n_2,n_3}^{(n_1)}=\lim_{n_1} \gamma_{n_1}=g$$ 
strongly in $L^1(\Omega;\R{d})$, which proves part (i).
On the other hand,
$$\lim_{n_2}\lim_{n_3} u_{n_1,n_2,n_3} = \gamma_{n_1} \in SBV(\Omega;\R{d}) 
\;\mbox{ and }
\lim_{n_1}\nabla \gamma_{n_1} =G_1
\mbox{ weak } * \mbox{ in } \cM(\Omega;\R{d\times N}),$$ and
$$\lim_{n_3} u_{n_1,n_2,n_3} = g_{n_2}^{(n_1)} \in SBV(\Omega;\R{d})
\; \mbox{ and }
\lim_{n_1} \lim_{n_2}\nabla g_{n_2}^{(n_1)} =
\lim_{n_1} \Gamma_{n_1}=G_2
\mbox{ weak } * \mbox{ in } \cM(\Omega;\R{d\times N}),$$
so part (ii) is proved.
Finally,
$$\lim_{n_1} \lim_{n_2} \lim_{n_3}
\nabla u_{n_1,n_2,n_3}
=\lim_{n_1} \lim_{n_2}\lim_{n_3} 
\nabla v_{n_2,n_3}^{(n_1)}=\lim_{n_1} \Upsilon_{n_1}=G_3
\mbox{ weak } * \mbox{ in } \cM(\Omega;\R{d\times N}),$$ which proves part (iii).

In addition, if $p >1$,
$$\sup_{n_1,n_2,n_3} \|\nabla u_{n_1,n_2,n_3}\|_{L^p(\Omega;\R{d \times N})} 
= \sup_{n_1,n_2,n_3} \|\nabla v_{n_2,n_3}^{(n_1)}\|_{L^p(\Omega;\R{d \times N})}
\leq \sup_{n_1} \|\Upsilon_{n_1}\|_{L^p(\Omega;\R{d \times N})} < + \infty.$$
\end{proof}

We now derive the counterpart of Proposition~\ref{T003}. 

\begin{proposition}\label{T004bis}
Under Assumptions \ref{gen_ass}, for every $(g,G_1,G_2,G_3)\in SD_{3,p}(\Omega)$,
	\begin{equation}\label{thesis}
		I_{3,p} (g, G_1, G_2,G_3) = \widetilde{I}_{3,p} (g, G_1, G_2,G_3) .
	\end{equation}
\end{proposition}
\begin{proof}
Given $\delta > 0$, let $(\gamma_{n_1},\Gamma_{n_1},\Upsilon_{n_1}) \in SD_{2,p}(\Omega)$ be such that 
\begin{equation}\label{neededlater} 
\gamma_{n_1}\weakstH (g,G_{1}), \qquad \Gamma_{n_1} \weakst G_2\,, \qquad\text{and}\quad \Upsilon_{n_1} \weakst G_3\,,
\end{equation}
and
\begin{equation}\label{eq1}
\widetilde{I}_{3,p}(g,G_1,G_2,G_3) + \delta \geq \liminf_{n_1}\widetilde I_{2,p}(\gamma_{n_1},\Gamma_{n_1}, \Upsilon_{n_1})=\liminf_{n_1} I_{2,p}(\gamma_{n_1},\Gamma_{n_1}, \Upsilon_{n_1}),
\end{equation}
where we invoked Proposition~\ref{T003} for the last equality.
Recalling \eqref{Ip2}, let $\big(u^{(n_1)}_{n_2,n_3}\big)\subset SBV(\Omega;\R{d})$ be such that $\big(u^{(n_1)}_{n_2,n_3}\big)\weakstH(\gamma_{n_1},\Gamma_{n_1},\Upsilon_{n_1})$ and 
$$I_{2,p}(\gamma_{n_1},\Gamma_{n_1}, \Upsilon_{n_1})\geq \liminf_{n_2}\liminf_{n_3} E\big(u^{(n_1)}_{n_2,n_3}\big)-\delta,$$
so that 
\begin{equation}\label{almostthere}
\widetilde{I}_{3,p}(g,G_1,G_2,G_3) + 2\delta \geq \liminf_{n_1}\liminf_{n_2}\liminf_{n_3} E\big(u^{(n_1)}_{n_2,n_3}\big).
\end{equation}
By \eqref{neededlater}, the sequence $u_{n_1,n_2,n_3}\coloneqq u^{(n_1)}_{n_2,n_3}$ converges to $(g,G_1,G_2,G_3)$ according to Definition~\ref{S000}, so that, by taking the infimum over all such sequences in \eqref{almostthere}, we obtain
$$\widetilde{I}_{3,p}(g,G_1,G_2,G_3) + 2\delta \geq I_{3,p}(g,G_1,G_2,G_3),$$
and we obtain the $\leq$ inequality in \eqref{thesis} by taking the limit $\delta\to0^+$.

To prove the reverse inequality, notice that for any $\delta>0$, we find a sequence $\big(u_{n_1,n_2,n_3}\big)\subset SBV(\Omega;\R{d})$ such that $u_{n_1,n_2,n_3}\weakstH (g,G_1,G_2,G_3)$ in the sense of Definition~\ref{S000} and
$$I_{3,p}(g,G_1,G_2,G_3)+\delta \geq \liminf_{n_1}\liminf_{n_2}\liminf_{n_3} E\big(u_{n_1,n_2,n_3}\big).$$
In particular, we have that 
$$\lim_{n_2}\lim_{n_3}u_{n_1,n_2,n_3}= \gamma_{n_1}\,, \qquad \lim_{n_2}\nabla\Big(\lim_{n_3} u_{n_1,n_2,n_3}\Big)= \Gamma_{n_1}\,,\qquad\text{and}\quad \lim_{n_2}\lim_{n_3}\nabla u_{n_1,n_2,n_3}=\Upsilon_{n_1}\,,$$
where $(\gamma_{n_1},\Gamma_{n_1},\Upsilon_{n_1})$ belongs to $SD_{2,p}(\Omega)$, for every $n_1$\,, and satisfies the convergences in \eqref{neededlater}.
Therefore, by taking the infimum over all sequences that satisfy \eqref{neededlater}, we can continue with the chain of inequalities and get 
$$\geq \liminf_{n_1} I_{2,p}(\gamma_{n_1},\Gamma_{n_1},\Upsilon_{n_1})= \liminf_{n_1} \widetilde{I}_{2,p} (\gamma_{n_1},\Gamma_{n_1},\Upsilon_{n_1})\geq \widetilde{I}_{3,p}(g,G_1,G_2,G_3),$$
where the equality holds owing to Proposition~\ref{T003}, and the last inequality is obtained upon taking the infimum over all the sequences $(\gamma_{n_1},\Gamma_{n_1},\Upsilon_{n_1})\in SD_{2,p}(\Omega)$ satisfying \eqref{neededlater} (see \eqref{T003eq}).
Therefore, we have obtained that $I_{3,p}(g,G_1,G_2,G_3)+\delta \geq \widetilde{I}_{3,p}(g,G_1,G_2,G_3)$, and we can conclude thanks to the arbitrariness of~$\delta$.
%
\end{proof}

\bigskip\bigskip\bigskip
\color{black}

\subsection*{Acknowledgements} 
MM and EZ are members of the \emph{Gruppo Nazionale per l'Analisi Matematica, la Probabilit\`{a} e le loro Applicazioni} (GNAMPA) of the \emph{Istituto Nazionale di Alta Matematica} (INdAM). 
The research of ACB was partially supported by National Funding from FCT - Funda\c{c}\~{a}o para a Ci\^{e}ncia e a Tecnologia through the Center for Mathematical Studies, University of Lisbon (CEMS.UL), project UID/04561/2025.
The research of JM was supported by GNAMPA, \emph{Programma Professori Visitatori}, year 2022, by FCT/Portugal through project UIDB/04459/2020 with DOI identifier 10-54499/UIDP/04459/2020 and by the FCT outgoing mobility program 2025.
He also gratefully acknowledges the support and hospitality of Sapienza-University of Rome  through the \emph{Programma Professori Visitatori}, year 2023 and  in May-June 2025.
The work of MM is partially supported by the MIUR project \emph{Dipartimenti di Eccellenza 2018-2022} (CUP E11G18000350001), by the \emph{Starting grant per giovani ricercatori} of Politecnico di Torino, and by the \emph{Geometric-Analytic Methods for PDEs and Applications} (GAMPA) project (CUP E53D23005880006), funded by European Union -- Next Generation EU within the PRIN 2022 program (D.D. 104 - 02/02/2022 Ministero dell'Universit\`{a} e della Ricerca). This manuscript reflects only the authors' views and opinions and the Ministry cannot be considered responsible for them.
EZ acknowledges the support of \emph{Mathematical Modelling of Heterogeneous Systems (MMHS)}, financed by the European Union - Next Generation EU, 
CUP B53D23009360006, Project Code 2022MKB7MM, PNRR M4.C2.1.1.
She also acknowledges partial funding from the GNAMPA Project 2023 \emph{Prospettive nelle scienze dei materiali: modelli variazionali, analisi asintotica e omogeneizzazione}.
 The work of EZ is also supported by Sapienza, University of Rome through
 Progetti di ricerca piccoli 2022: Asymptotic Analysis for Composites, Fractured Materials and with Defects

\bibliographystyle{plain}

\end{document}